\documentclass{arxiv} 

\newif\ifarxiv
\arxivtrue

\usepackage{enumitem}
\usepackage{subcaption}

\usetikzlibrary{external}
\tikzexternaldisable

\usetikzlibrary{intersections}
\pgfplotsset{
    colormap = {gradient}{
        color = (red!90!black) color = (red) color = (yellow) color = (blue) color = (blue!90!black)
    },
}

\newcommand{\smep}[2][i]{\mcl{B}_{#1}\mleft(#2\mright)}
\newcommand{\res}{\vc{r}}
\newcommand{\bound}{\mt{E}}
\newcommand{\err}[1]{\Delta #1}
\newcommand{\backward}[1]{\eta(#1)}
\newcommand{\cond}[1]{\kappa(#1)}
\newcommand{\per}[1]{\mcl{\widetilde{A}}\mleft(#1\mright)}
\renewcommand{\norm}[2][2]{\Vert#2\Vert_{#1}}

\title{Rectangular Multispectral Perturbation Theory\thanks{The work was supported in part by the KU Leuven: Research Fund (grants IBOF/23/064, C3/20/117, and C3I/21/00316); in part by the Fonds Wetenschappelijk Onderzoek (grants S005319 and T001919N); in part by the Departement Economie, Wetenschap \& Innovatie via the Flanders AI Research Program; in part by the Vlaams Agentschap Innoveren \& Ondernemen (grant HBC/2021/0076); and in part by the European Research Council (grant 885682).}}

\author[$\dagger$, $\ddagger$]{Christof Vermeersch}
\author[$\ddagger$]{Sarthak De}
\author[$\ddagger$]{Bart De Moor}

\affil[$\dagger$]{Corresponding author (\myemail)}
\affil[$\ddagger$]{Center for Dynamical Systems, Signal Processing, and Data Analytics (STADIUS), Dept. of Electrical Engineering (ESAT), KU Leuven, Kasteelpark Arenberg 10, 3001 Leuven, Belgium}

\begin{document}
    \maketitle
    
    \begin{abstract}
    We provide a first systematic treatment of so-called rectangular multispectral perturbation theory.
    With their paper from 2003, Hochstenbach and Plestenjak [``Backward Error, Condition Numbers, and Pseudospectra for the Multiparameter Eigenvalue Problem'' in Linear Algebra and its Applications] extended perturbation theory from one-parameter eigenvalue problems to multiple spectral parameters. 
    After two decades, we take it one step further and consider a different manifestation of the multiparameter eigenvalue problem that consists of one matrix equation with rectangular coefficient matrices.
    We perform a norm-wise backward error analysis, define condition numbers for both eigenvalues and eigenvectors, and introduce the pseudospectrum while also considering the computational implications of working with multiple spectral parameters.
    The rectangular shape hampers a direct application of the existing definitions and properties.
    For example, the left null space at a given eigenvalue is non-trivial and the dimensions of the left and right eigenvectors are different.
    Through numerical examples, we illustrate and link the different concepts from the perturbation theory.
    A system identification application seem to suggest that, in optimization-driven problems for which multiparameter reformulations exist, the globally optimal solutions tend to coincide with the best-conditioned eigenvalues.
\end{abstract}

\section{Introduction}

    When we want to know whether the computed eigenvalues of a rectangular multiparameter eigenvalue problem are satisfactory, stability and conditioning are two key concepts that come into the picture~\cite{trefethen1997numerical}.
    Consider the following example.
    \begin{example}
        \label{ex:conditionandstability}
        If we have rectangular coefficient matrices $\mat_0$, $\mat_1$, and $\mat_2$, then solving $\mep{\eig} \rvec = \left(\mat_0 + \eigcomp[1] \mat_1 + \eigcomp[2] \mat_2\right) \rvec = \vc{0}$ results in one or multiple eigenvalue solutions $\eig = \left(\eigcomp[1], \eigcomp[2]\right)$.
        In exact arithmetic we obtain an exact solution $\sol{\eig}$, and this is the end of the story.
        However, on computers, if we want to compute the solution of $\mep{\eig} \rvec = \vc{0}$ in floating-point arithmetic, the intermediate results are replaced by nearby machine numbers, causing rounding errors in the computed solution $\comp{\eig}$.
        Moreover, the values of the coefficient matrices cannot be represented exactly on the computer, so their entries are already replaced by machine numbers before the computations even start.
        So, we hope that our numerical algorithm computes a good approximation $\comp{\eig}$ of the true solution $\sol{\eig}$, in the sense that $\mep{\comp{\eig}} \comp{\rvec} \approx \vc{0}$, or even better $\comp{\eig} \approx \sol{\eig}$.
    \end{example}

    It is clear that the exact solution to the problem may be very sensitive to perturbations of the coefficient matrices and also the numerical algorithm may be influenced greatly by perturbations.
    A good approach for deciding whether the numerical algorithm performs well is to consider a slightly perturbed version of the problem. 
    Perturbation theory aims at deriving (backward) error expressions and condition numbers. 
    Backward errors help to assess the quality of the obtained numerical solutions and stability of the numerical algorithms, while condition numbers explain the sensitivity to perturbations in the input data~\cite{trefethen1997numerical}.
    Pseudospectra can help to reveal the conditioning of eigenvalues.
    \emph{The paper's goal is to provide a first systematic picture of rectangular multispectral perturbation theory, by performing a norm-wise backward error analysis, defining condition numbers for eigenvalues and eigenvectors, and introducing the pseudospectrum of rectangular multiparameter eigenvalue problems.}

    \subsection{Rectangular multiparameter eigenvalue problems}

        We focus on the (linear) rectangular multiparameter eigenvalue problem, which is defined as follows.
        \begin{definition}
            \label{def:mep}
            Given coefficient matrices $\mat_i \in \Cset^{k \times \ell}$ (with $k = \ell + m - 1$) that lead to a full normal rank ``multiparameter matrix pencil'' $\mep{\eig}$, the \emph{rectangular multiparameter eigenvalue problem (rMEP)} consists in finding all $m$-tuples $\eig = (\eigcomp[1], \eigcomp[2], \ldots, \lambda_m) \in \Cset^m$ and corresponding vectors $\rvec \in \Cset^\ell_0$, so that
            \begin{equation}
                \label{eq:mep}
                \mep{\eig} \rvec = \mat_0 \rvec + \sum_{i = 1}^m \eigcomp[i] \mat_i \rvec = \sum_{i = 0}^m \eigcomp[i] \mat_i \rvec = \vc{0},
            \end{equation}
            in which setting $\eigcomp[0] = 1$ allows for a more compact notation of the problem.
        \end{definition}

        \begin{remark}
            We restrict our investigation to the zero-dimensional components of the solution set. 
            The size condition on the coefficient matrices is a necessary (but not a sufficient) condition in order for the rMEP to have a zero-dimensional solution set: there are $k$ equations and one non-triviality constraint on $\rvec$ (e.g., $\norm{\rvec} = 1$) in $\ell + m$ unknowns ($\ell$ elements in the eigenvectors $\rvec$ and $m$ eigenvalues), thus $k + 1 \geq \ell + m$.
            Notice that for one-parameter eigenvalue problem, this condition boils down to having square coefficient matrices.
        \end{remark}

        An $m$-tuple $\sol{\eig}$ that satisfies~\eqref{eq:mep} is an eigenvalue. 
        Together with the associated eigenvector $\sol{\rvec}$, the eigenvalue forms an eigenpair $(\sol{\eig}, \sol{\rvec})$ of the rMEP.
        Another way to phrase the rMEP is by considering the eigenvalues for which the rank drops below the normal rank of the problem.
        \begin{corollary}
            \label{cor:rankdefinition}
            Given the problem in~\cref{def:mep}, the tuple $\sol{\eig} \in \Cset^m$ is an eigenvalue if
            \begin{equation}
                \sol{\rho} = \rank\mleft(\mep{\sol{\eig}}\mright) < \nrank\mleft(\mep{\eig}\mright), 
            \end{equation}
            where the normal rank $\rho = \nrank\mleft(\mep{\eig}\mright)$ is the rank attained at a generic point $\eig$.
        \end{corollary} 

        Next to the (right) eigenvector $\sol{\rvec}$ associated to an eigenvalue $\sol{\eig}$, there also exist vectors that solve the matrix equation $\lvec^\herm \mep{\eig} = \vc{0}$.
        Due to the rectangular matrix size of $\mep{\sol{\eig}}$, there are $\ell + m - \sol{\rho} - 1$ vectors in the left null space of $\mep{\sol{\eig}}$, of which $m - 1$ can be considered to be part of the trivial left null space that exists for any $\eig$.
        More information about the left null space and left eigenvectors is given in \cref{app:leftnullspace}.

        \begin{definition}
            \label{def:secularequations}
            Let $\sigma_1, \sigma_2, \ldots, \sigma_L$ denote the $L = \binom{k}{\ell}$ selections of $\ell$ rows from a set of $k$ rows.
            If we consider the multivariate linear matrix $\mep{\eig}$ as defined in \cref{def:mep}, then the \emph{secular equations} are given by
            \begin{equation}
                \label{eq:secularequations}
                \chi_i(\eig) = \det\left[\mep{\eig}\right]_{\sigma_i} = 0, \quad i = 1, 2, \ldots, L. 
            \end{equation}
            The common roots of the secular polynomials in~\eqref{eq:secularequations} correspond to the solutions of the rMEP.
        \end{definition}
        
        \begin{example}
            \label{ex:runningexample}
            A running example that we use throughout the paper is the two-parameter eigenvalue problem
            \begin{equation}
                \mep{\eig} \rvec = \left(\mat_0 + \eigcomp[1] \mat_1 + \eigcomp[2] \mat_2\right) \rvec = \vc{0}, 
            \end{equation}
            with coefficient matrices
            \begin{equation}
                \mat_0 = -
                \begin{bmatrix}
                    6 & 4 \\
                    0 & 2 \\
                    2 & 2 
                \end{bmatrix}, 
                \mat_1 = 
                \begin{bmatrix}
                    2 & 1 \\
                    0 & 0 \\
                    2 & 0 
                \end{bmatrix}, \eqand
                \mat_2 = 
                \begin{bmatrix}
                    2 & 1 \\
                    0 & 2 \\
                    0 & 2 
                \end{bmatrix}
            \end{equation}
            The secular equations of this problem are (after simplification) given by
            \begin{equation}
                \begin{aligned}
                    \chi_1(\eig) &= 1 - \eigcomp[1] - \eigcomp[2] + \eigcomp[1] \eigcomp[2] = 0, \\
                    \chi_2(\eig) &= 2 + 3 \eigcomp[1] - 7 \eigcomp[2] - \eigcomp[1]^2 + \eigcomp[1] \eigcomp[2] + 2 \eigcomp[2]^2 = 0, \\
                    \chi_3(\eig) &= 3 + \eigcomp[1] + 4 \eigcomp[2] - \eigcomp[1] \eigcomp[2] - \eigcomp[2]^2 = 0, 
                \end{aligned}
            \end{equation}
            and its three eigenvalues are $\sol{\eig}_{(1)} = (1, 2)$, $\sol{\eig}_{(2)} = (3, 1)$, and $\sol{\eig}_{(3)} = (1, 1)$.
            All are visualized in \cref{fig:runningexample}.
        \end{example} 

        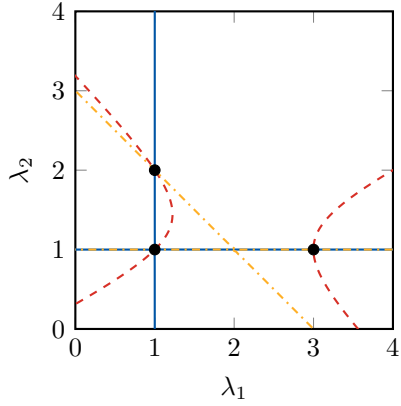
\begin{figure}
            \centering
            \tikzexternalenable
            \begin{tikzpicture}[trim axis left, trim axis right]
    \begin{axis}[
        figurestyle,
        xmin= 0,
        xmax= 4,
        ymin= 0,
        ymax= 4,
        xlabel = {$\eigcomp[1]$},
        ylabel = {$\eigcomp[2]$}
        ]

        \addplot +[blueline, no markers,
            raw gnuplot,
            name path = P
            ] gnuplot {
            set contour base;
            set cntrparam levels discrete 0.0001;
            unset surface;
            set view map;
            set isosamples 500;
            set samples 500;
            splot 1 - x - y + x*y;
        };
        \addlegendimage{blueline, no markers}
        \label{plot:chi1}

        \addplot +[redline, dashed, no markers,
            raw gnuplot,
            name path = Q
            ] gnuplot {
            set contour base;
            set cntrparam levels discrete 0.003;
            unset surface;
            set view map;
            set isosamples 500;
            set samples 500;
            splot 2 + 3*x - 7*y - x**2 + x*y + 2*y**2;
        };
        \addlegendimage{redline, dashed, no markers}
        \label{plot:chi2}

        \addplot +[yellowline, dashdotted, no markers,
            raw gnuplot,
            name path = Q
            ] gnuplot {
            set contour base;
            set cntrparam levels discrete 0.0003;
            unset surface;
            set view map;
            set isosamples 500;
            set samples 500;
            splot 3 - x - 4*y + x*y + y**2;
        };
        \addlegendimage{yellowline, dashdotted, no markers}
        \label{plot:chi3}

        \filldraw[black] (1, 2) circle (2pt) {};
        \filldraw[black] (3, 1) circle (2pt) {};
        \filldraw[black] (1, 1) circle (2pt) {};
    \end{axis}
\end{tikzpicture}
            \tikzexternaldisable
            \caption{Real picture of the secular equations $\chi_1(\eig) = 0$ (\ref{plot:chi1}), $\chi_2(\eig) = 0$ (\ref{plot:chi2}), and $\chi_3(\eig) = 0$ (\ref{plot:chi3}) of the two-parameter eigenvalue problem in \cref{ex:runningexample}. The intersections correspond to the three eigenvalues.}
            \label{fig:runningexample}
        \end{figure}

    \subsection{Backward errors, condition numbers, and pseudospectra}

        Backward errors reveal the stability of a numerical algorithm.
        The problem's condition number measures the sensitivity of the output with respect to perturbations in the input. 
        In the case of solving linear systems, the idea of stability and condition dates back to the work of von Neumann, Goldstine, and Turing, the last one coining the term for the first time.
        Wilkinson, later on, systematized and popularized condition numbers as part of his broader work on rounding error analysis and numerical stability.
        For an algorithm computing the output in finite precision arithmetic the importance of the condition number, $\cond{\cdot}$, stems from the ``rule of thumb'' that the forward error is smaller than $\cond{\cdot}$ times the backward error~\cite{higham2002accuracy}. 
        Problem independent definitions of stability and conditioning, from which we start in the paper, can be found in \cref{app:perturbationtheory}.

        The theory of backward errors and condition numbers is well developed for one-parameter eigenvalue problems, with various results for both generalized~\cite{stewart1990matrix,higham1998structured,fraysse1998note} and polynomial eigenvalue problems~\cite{dedieu2003perturbation,tisseur2000backward}.
        For the more difficult case of singular matrix polynomials, more recent results exist, with the introduction of so-called weak condition numbers~\cite{lotz2020wilkinson}.
        For these problems, classical perturbation theory fails to predict the observed accuracy of computed solutions and weak condition numbers can come to the rescue.
        Hochstenbach and Plestenjak~\cite{hochstenbach2003backward} provided backward error expressions and condition numbers for coupled square multiparameter eigenvalue problems, which consist of multiple multivariate matrix polynomials.

        The pseudospectrum is another tool to study the sensitivity of eigenvalues to perturbations.
        It provides an analytical and graphical alternative to investigate eigenvalue problems.
        The concept has been independently invented in different fields, with some early definitions going back to Landau in the 1970s.
        Pseudospectra have been studied for all one-parameter eigenvalue problems~\cite{tisseur2001structured,trefethen1999computation,trefethen2005spectra} and the coupled square multiparameter eigenvalue problem~\cite{hochstenbach2003backward}.
        Wright and Trefethen considered pseudospectra of rectangular univariate matrix polynomials in~\cite{wright2002pseudospectra}, from which we take some inspiration on how to deal with rectangular matrices in our work.


    \subsection{Outline and contributions}

        We want to quantify the difficulty of computing eigenvalues and the quality of the obtained solutions, answering the questions raised in~\cref{ex:conditionandstability}.
        We restrict our investigation to isolated, finite, and simple solutions, and, although most results hold for polynomial problems, we focus explicitly on the linear rMEP, mostly for notational clarity.

        We begin in \cref{sec:error} with a norm-wise backward error analysis, \emph{providing closed-form expressions to assess the quality of a computed solution}.
        In \cref{sec:condition}, we shift to the forward sensitivity of the problem and \emph{derive condition numbers for both eigenvalues and eigenvectors}.
        \Cref{sec:pseudospectrum} considers the \emph{theoretical and computational analysis of the pseudospectrum} of an rMEP.
        Instead of adding a separate section, we provide numerical examples throughout the different sections to illustrate our developments.
        In~\cref{sec:application}, we bring all three concepts together on a system identification example---an \emph{incidental finding there suggests that, in an optimization-driven system identification problem, globally optimal solutions tend to coincide with the best-conditioned eigenvalues}.
        We conclude in \cref{sec:conclusion} with final remarks and open research questions.

        The work is supplemented with \matlab implementations\footnote{You can download the \matlab functions (\function{mperr}, \function{mpcond}, and \function{mppseudo}) and examples used throughout the paper from \mywebsite.} to compute the norm-wise backward error, condition number, and pseudospectrum.
        Two appendices provide additional information: \cref{app:leftnullspace} reviews the left null space and left eigenvectors of the rMEP and \cref{app:perturbationtheory} collects abstract definitions from perturbation theory.

\section{Backward error analysis}
    \label{sec:error}

    A $k \times \ell$ coefficient matrix $\mat_i$ can be perturbed by a matrix $\err{\mat_i} \in \Cset^{k \times \ell}$, resulting in the perturbed coefficient matrix $\mat_i + \err{\mat_i}$.
    The matrix $\err{\mat_i}$ is bounded norm-wise by an error matrix $\bound_i \in \Cset^{k \times \ell}$. 
    This means that $\norm{\err{\mat_i}} \leq \epsilon \norm{\bound_i}$. 
    Via a particular choice of $\bound_i$, the perturbations are considered absolute ($\bound_i = \mt{1}/\sqrt{k \ell}$) or relative ($\bound_i = \mat_i$).\footnote{The specific choice of $\bound_i$ for absolute perturbations differs from the traditional diagonal matrix in the square matrix case. By choosing these specific perturbations, we guarantee that the norm is one.} 
    We have thus a second, perturbed rMEP next to the original rMEP in~\eqref{eq:mep}, with perturbed coefficient matrices:
    \begin{equation}
        \label{eq:perturbedmep}
        \per{\eig} \rvec = \sum_{i = 0}^m \eigcomp[i] \left(\mat_i + \err{\mat_i}\right) \rvec = \vc{0}, 
    \end{equation}
    in which setting $\eigcomp[0]= 1$ allows again for a more compact formulation.

    We establish in this section the norm-wise\footnote{It is also possible to consider element-wise perturbations in the analysis. The authors expect that a generalization of the results for the GEP in~\cite{higham1998structured} to the rMEP is possible.} backward error for a finite eigenvalue and its corresponding eigenvector (\cref{sec:eigenpairerror}).
    We also consider the optimal backward error associated with the eigenvalue only (\cref{sec:optimalerror}).\footnote{Numerical algorithms to compute both types of backward errors are implemented in \function{mperr}.}

    \subsection{Eigenpair backward error}
        \label{sec:eigenpairerror}

        We define the backward error for an approximate eigenpair as follows:
        \begin{definition}
            Let $(\comp{\eig}, \comp{\rvec}) \in \Cset^m \times \Cset^\ell$ with $\norm{\comp{\rvec}} = 1$ be an approximate eigenpair of~\eqref{eq:mep}.
            The \emph{norm-wise backward error of an eigenpair} $\backward{\comp{\eig}, \comp{\rvec}}$ is given by 
            \begin{equation}
                \begin{gathered}
                    \backward{\comp{\eig}, \comp{\rvec}} = \min_{\err{\mat_i}} \epsilon, \\
                    \sub \per{\comp{\eig}} \comp{\rvec} = \vc{0} \eqand \norm{\err{\mat_i}} \leq \epsilon \norm{\bound_i}, \quad i = 0, 1, \ldots, m.
                \end{gathered}
            \end{equation}
            The choice of error matrices $\mt{E}_i$ determines whether the norm-wise backward error is absolute or relative.
        \end{definition}

        The following theorem, which is a generalization of~\cite{fraysse1998note,tisseur2000backward} for $m > 1$, gives a computable expression for $\backward{\comp{\eig}, \comp{\rvec}}$.
        \begin{theorem}
            \label{thm:backwarderror}
            Consider an rMEP $\mep{\eig} \rvec = \vc{0}$ and an approximate eigenpair $(\comp{\eig}, \comp{\rvec}) \in \Cset^m \times \Cset^\ell$.
            The norm-wise backward eigenpair error can be computed as
            \begin{equation}
                \label{eq:computedbackwarderror}
                \backward{\comp{\eig}, \comp{\rvec}} = \frac{\norm{\comp{\res}}}{\comp{\gamma} \norm{\comp{\rvec}}},
            \end{equation}
            where $\comp{\res} = \mep{\comp{\eig}} \comp{\rvec}$ is the residual at $(\comp{\eig}, \comp{\rvec})$ and $\comp{\gamma} = \norm{\bound_0} + \sum_{i = 1}^m \abs{\comp{\eigcomp[i]}} \norm{\bound_i}$.
        \end{theorem} 

        \begin{proof}
            From $\per{\comp{\eig}} \comp{\rvec} = \vc{0}$, it follows that $\comp{\res} = - \err{\mat_0} \comp{\rvec} - \sum_{i = 1}^m \comp{\eigcomp[i]} \err{\mat_i} \comp{\rvec}$. 
            Applying the feasibility constraints $\norm{\err{\mat_i}} \leq \epsilon \norm{\bound_i}$ gives
            \begin{align}
                \norm{\comp{\res}} & = \norm{\err{\mat_0} \comp{\rvec} + \sum_{i = 1}^m \comp{\eigcomp[i]} \err{\mat_i} \comp{\rvec}} \\
                & \leq \left(\norm{\err{\mat_0}} + \sum_{i = 1}^m \abs{\comp{\eigcomp[i]}} \norm{\err{\mat_i}}\right) \norm{\comp{\rvec}} \\
                & \leq \left(\norm{\bound_0} + \sum_{i = 1}^m \abs{\comp{\eigcomp[i]}} \norm{\bound_i}\right) \epsilon \norm{\comp{\rvec}} = \comp{\gamma} \epsilon \norm{\comp{\rvec}}, 
            \end{align}
            which shows that the expression in~\eqref{eq:computedbackwarderror} is a lower bound for $\backward{\comp{\eig}, \comp{\rvec}}$.
            The lower bound can be attained for the perturbations 
            \begin{gather}
                \err{\mat_0} = \frac{- \norm{\bound_0}}{\comp{\gamma}} \comp{\res} \vc{w}^\herm, \\
                \err{\mat_i} = \frac{- \sign\mleft(\comp{\eigcomp[i]}\mright) \norm{\bound_i}}{\comp{\gamma}} \comp{\res} \vc{w}^\herm, 
            \end{gather}
            in which the vector $\vc{w} \in \Cset^\ell$ forms a dual pair with $\comp{\rvec}$, such that $\norm{\comp{\res} \vc{w}^\herm} = \norm{\comp{\res}} \norm{\vc{w}}^\vee = \norm{\comp{\res}} / \norm{\comp{\rvec}}$ (the $l_2$-norm is self-dual).
            For a complex eigenvalue $\eigcomp[i] \in \Cset$, its sign is defined as 
            \begin{equation}
                \sign\mleft(\eigcomp[i]\mright) = 
                \left\lbrace
                \begin{aligned}
                    & \bar{\eigcomp[i]} / \abs{\eigcomp[i]}, & \quad \eigcomp[i] \neq 0, \\
                    & 0, & \quad \eigcomp[i] = 0.
                \end{aligned}
                \right.
            \end{equation}
            The lower bound can be verified by substituting these perturbations into the feasibility constraints.
        \end{proof}

        The norm-wise backward eigenpair error becomes zero when the residual is zero.
        Suppose that the denominator is also zero in that case, we still take $\backward{\eig, \rvec} = 0$ as the eigenpair is an exact one.
        When the residual is not zero, a zero denominator corresponds to forbidding perturbations on the coefficient matrices; the norm-wise backward eigenpair error then beign infinite.

        When $m = 1$, the well-known one-parameter case of the GEP, \cref{thm:backwarderror} reduces to the results in~\cite{fraysse1998note,higham1998structured} with $\backward{\comp{\eig}, \comp{\rvec}} = \norm{\mat_0 + \comp{\lambda} \mat_1} / \left\lbrack(\norm{\err \mat_0} + \abs{\comp{\lambda}} \norm{\err \mat_1}) \norm{\comp{\rvec}}\right\rbrack$.
        If we consider computed right eigenvectors $\comp{\rvec}$ that are normalized, the norm in the denominator of \cref{thm:backwarderror} disappears.
        Then it is very similar to the multiparameter results in~\cite[Theorem~2]{hochstenbach2003backward} when simplifying the expression for a single matrix equation.

    \subsection{Optimal backward error}
        \label{sec:optimalerror} 

        In many situations only the eigenvalues are of interest, and the eigenvectors are not computed during that process.
        The notion of optimal backward error allows judging the numerical qualities of the eigenvalues without knowing the associated eigenvector.
        If we are only interested in the eigenvalues of the rMEP, then a more appropriate measure is the (norm-wise) backward error for a given eigenvalue, given the best possible associated eigenvector. 
        This results in the norm-wise backward eigenvalue error.
        \begin{definition}
            Let $\comp{\eig} \in \Cset^m$ be an approximate eigenvalue of~\eqref{eq:mep}.
            The \emph{norm-wise backward eigenvalue error} is defined as 
            \begin{equation}
                \begin{gathered}
                    \backward{\comp{\eig}} = \min_{\rvec} \backward{\comp{\eig}, \rvec}, \\
                    \sub \norm{\rvec} = 1,
                \end{gathered}
            \end{equation}
            minimizing over all possible eigenvectors.
        \end{definition}

        Notice that the set $\left\lbrace\comp{\eig} : \backward{\comp{\eig}} \leq \epsilon\right\rbrace$ for a certain $\epsilon > 0$ will turn out to be the $\epsilon$-pseudospectrum in \cref{sec:pseudospectrum} (see \cref{thm:equivalences}).
        We can determine the norm-wise backward eigenvalue error via the following theorem.
        \begin{theorem}
            Consider a rectangular matrix polynomial $\mep{\eig} \in \Cset^{k \times \ell}$ and a point $\comp{\eig} \in \Cset^m$.
            The norm-wise backward eigenvalue error can be computed as
            \begin{equation}
                \backward{\comp{\eig}} = \frac{\sigma_\mathrm{min}\mleft(\mep{\comp{\eig}}\mright)}{\comp{\gamma}},
            \end{equation}
            with $\sigma_\mathrm{min}\mleft(\cdot\mright)$ the smallest singular value of a matrix.
        \end{theorem}
        \begin{proof}
            We consider the minimization of~\eqref{eq:computedbackwarderror}, which is $\min_{\rvec} \backward{\comp{\eig}, \rvec}$, over a normalized vector $\rvec$ and use 
            \begin{equation}
                \min_{\norm{\vc{x}} = 1} \norm{\vc{A} \vc{x}} = \sigma_\mathrm{min}\mleft(\mt{A}\mright),
            \end{equation}
            which is a standard property of the singular value decomposition~\cite{golub2013matrix}.
        \end{proof}

        \begin{example}
            Consider the running example from \cref{ex:runningexample}.
            One of the solutions is $\sol{\eig} = (1, 1)$ with eigenvector $\sol{\rvec} = \begin{bmatrix} -\sqrt{2} / 2 & \sqrt{2} / 2 \end{bmatrix}^\tr$.
            We take an approximation of the eigenvalue $\comp{\eig} = (\num{0.9999}, \num{0.9999})$ and eigenvector $\comp{\rvec} = \begin{bmatrix} -0.7070 & 0.7072 \end{bmatrix}^\tr$.
            The norm of the residual is $\norm{\comp{\res}} = \num{2.9e-4}$. 
            The norm-wise backward eigenpair error is $\backward{\comp{\eig}, \comp{\rvec}} = \num{2.1e-5}$, while the norm-wise backward eigenvalue error is as expected lower, namely $\backward{\comp{\eig}} = \num{1.0e-5}$.
        \end{example}

\section{Two condition numbers}
    \label{sec:condition}


        Throughout the section, we fix a simple eigenpair $(\sol{\eig}, \sol{\rvec})$ of the rMEP and consider the map $f$,
        \begin{equation}
            \label{eq:fullmap}
            f : \left\lbrace \mat_i \right\rbrace_{i = 0}^m \to \left(\sol{\eig}, \sol{\rvec}\right),
        \end{equation}
        as the local branch sending coefficient matrices in a neighborhood of $\left\lbrace\mat_i\right\rbrace_{i = 0}^m$ to this eigenpair.
        The condition number measures the sensitivity to perturbations in $\mat_i$ of the map~\eqref{eq:fullmap} for a semi-simple eigenvalue $\sol{\eig}$ of the rMEP.
        Both absolute ($\bound_i = \mt{1} / \sqrt{k \ell}$) and relative ($\bound_i = \mat_i$) perturbations can be considered.

        We consider two types of condition numbers: considering a projection of the map onto the eigenvalues (\cref{sec:eigenvaluecondition}) and the eigenvectors (\cref{sec:eigenvectorcondition}).\footnote{The function \function{mpcond} implements \cref{thm:eigenvaluecond} to compute the eigenvalue condition number of an rMEP.}

    \subsection{Eigenvalue condition number}
        \label{sec:eigenvaluecondition}

        In order to obtain the eigenvalue condition number, we project the map~\eqref{eq:fullmap} onto the eigenvalue solution, i.e., 
        \begin{equation}
            f_{\pi_{\eig}} : \left\lbrace\mat_i\right\rbrace_{i = 0}^m \to \sol{\eig}.
        \end{equation}


        \begin{definition}
            \label{def:eigenvalueconditionnumber}
            Consider an rMEP $\mep{\eig} \rvec = \vc{0}$ and eigenvalue $\sol{\eig}$.
            A \emph{norm-wise relative eigenvalue condition number} can be defined by 
            \begin{equation}
                \begin{gathered}
                    \cond{\sol{\eig}} = \lim_{\epsilon \to 0} \sup \frac{\norm{\err{\eig}}}{\epsilon \norm{\sol{\eig}}}, \\
                    \sub \per{\sol{\eig} + \err{\eig}} \left(\sol{\rvec} + \err{\rvec}\right) = \vc{0}, \quad \norm{\err{\mat_i}} \leq \epsilon \norm{\bound_i}, \quad i = 0, 1, \ldots, m,
                \end{gathered}
            \end{equation}    
            together with the implicit requirement that $\err{\rvec} \to \vc{0}$ as $\epsilon \to 0$. 
        \end{definition}

        The above-mentioned implicit requirement is included to avoid the unwanted behavior of taking $\sol{\eig} + \err{\eig}$ equal to another eigenvalue than $\sol{\eig}$, which would lead to making it possible to set $\err{\mat_i} = 0$ and obtain $\cond{\sol{\eig}} = \infty$. 

        \begin{lemma}
            \label{lemma:genericity}
            Take $k = \ell + m - 1$.
            We define an auxiliary matrix $\mt{B}$ for an eigenpair $(\sol{\eig}, \sol{\rvec})$ of the rMEP with coefficient matrices $\mat_0, \mat_1, \ldots, \mat_m \in \Cset^{k \times \ell}$:
            \begin{equation}
                \label{eq:auxiliarymatrix}
                \mt{B} = 
                \begin{bmatrix}
                    \lvec_1^\herm \mat_1 \sol{\rvec} & \lvec_1^\herm \mat_2 \sol{\rvec} & \cdots & \lvec_1^\herm \mat_m \sol{\rvec} \\
                    \lvec_2^\herm \mat_1 \sol{\rvec} & \lvec_2^\herm \mat_2 \sol{\rvec} & \cdots & \lvec_2^\herm \mat_m \sol{\rvec} \\
                    \vdots & \vdots & \ddots & \vdots \\
                    \lvec_m^\herm \mat_1 \sol{\rvec} & \lvec_m^\herm \mat_2 \sol{\rvec} & \cdots & \lvec_m^\herm \mat_m \sol{\rvec}
                \end{bmatrix}.
            \end{equation}
            The vectors $\lvec_i \in \Cset^k$, for $i = 1, 2, \ldots, m$, are linearly independent vectors in the left null space of the $k \times \ell$ matrix $\mep{\sol{\eig}}$.
            If the eigenvalue $\sol{\eig}$ is algebraically simple, then $\mt{B}$ is nonsingular for any choice of basis vectors $\lvec_i \in \Cset$, for $i = 1, 2, \ldots, m$, of the left null space.
        \end{lemma}

        A more involved variant of \cref{lemma:genericity} that holds for polynomial rMEPs can be found in \cite{plestenjak2026homotopy}, of which \cref{lemma:genericity} is clearly a corollary.
        Below, we give a sketch of the proof of \cref{lemma:genericity} by restricting the proof in \cite{plestenjak2026homotopy} to linear problems.

        \begin{proof}
            Let $L = \binom{k}{\ell}$.
            An eigenvalue $\sol{\eig}$ is algebraically simple if it is a simple common root of the multivariate secular system with equations $\chi_i(\eig) = 0$, for $i = 1, 2, \ldots, L$ (see \cref{def:secularequations}).
            Consequently, the Jacobian of the multivariate secular system,
            \begin{equation}
                \begin{bmatrix}
                    \frac{\chi_1(\eig)}{\eigcomp[1]} & \frac{\chi_1(\eig)}{\eigcomp[2]} & \cdots & \frac{\chi_1(\eig)}{\eigcomp[m]} \\
                    \frac{\chi_2(\eig)}{\eigcomp[1]} & \frac{\chi_2(\eig)}{\eigcomp[2]} & \cdots & \frac{\chi_2(\eig)}{\eigcomp[m]} \\
                    \vdots & \vdots & \ddots & \vdots \\
                    \frac{\chi_L(\eig)}{\eigcomp[1]} & \frac{\chi_L(\eig)}{\eigcomp[2]} & \cdots & \frac{\chi_L(\eig)}{\eigcomp[m]}
                \end{bmatrix},
            \end{equation}
            has full rank.
            The $L$ secular equations $\chi_i(\eig) = \det\mleft[\mep{\eig}\mright]_{\sigma_i} = 0$ of an rMEP are the determinants of the selection of $\ell$ rows of the matrix polynomial.
            The row selection can be done by premultiplying the matrix polynomial with the matrix $\mt{P}_i \in \Nset^{k \times \ell}$ consisting of unit vectors, i.e., $\left[\mep{\eig}\right]_{\sigma_i} = \mt{P}_i^\tr \mep{\eig}$.
            Choosing $m$ such row selection matrices so that the corresponding rows of the Jacobian matrix are linearly independent results in a coupled square multiparameter eigenvalue problem, 
            \begin{equation}
                \label{eq:smep}
                \mt{P}_i^\tr \mep{\eig} \rvec = \vc{0}, \quad i = 1, 2, \ldots, m
            \end{equation}
            The eigenvalues of the coupled square multiparameter eigenvalue problem can be described as the common roots of a second system of secular equations, i.e., $\varphi_i(\eig) = \det\mleft(\mt{P}_i^\tr \mep{\eig}\mright) = 0$, for $i = 1, 2, \ldots, m$.
            The simple eigenvalue $\sol{\eig}$ that results in a full Jacobian matrix is also a simple eigenvalue of~\eqref{eq:smep} because of the specific choice of row selection matrices.

            Let $\vc{v}_1 \otimes \vc{v}_2 \otimes \cdots \otimes \vc{v}_m$ be the left eigenvector of the coupled square multiparameter eigenvalue problem in~\eqref{eq:smep}. 
            Since $\sol{\eig}$ is an algebraically simple eigenvalue, it follows from \cite[Lemma~3]{kosir1994finite} that the matrix 
            \begin{equation}
                \mt{B}' = 
                \begin{bmatrix}
                    \vc{v}_1^\herm \mt{P}_1^\tr \mat_1 \sol{\rvec} & \vc{v}_1^\herm \mt{P}_1^\tr \mat_2 \sol{\rvec} & \cdots & \vc{v}_1^\herm \mt{P}_1^\tr \mat_m \sol{\rvec} \\
                    \vc{v}_2^\herm \mt{P}_2^\tr \mat_1 \sol{\rvec} & \vc{v}_2^\herm \mt{P}_2^\tr \mat_2 \sol{\rvec} & \cdots & \vc{v}_2^\herm \mt{P}_2^\tr \mat_m \sol{\rvec} \\
                    \vdots & \vdots & \ddots & \vdots \\
                    \vc{v}_m^\herm \mt{P}_m^\tr \mat_1 \sol{\rvec} & \vc{v}_m^\herm \mt{P}_m^\tr \mat_2 \sol{\rvec} & \cdots & \vc{v}_m^\herm \mt{P}_m^\tr \mat_m \sol{\rvec}
                \end{bmatrix}
            \end{equation}
            is nonsingular.
            By introducing vectors $\lvec_i = \mt{P}_i \vc{v}_i$ from the left null space $\kernel\big(\mep{\sol{\eig}}^\herm\big)$, we can prove that the auxiliary matrix $\mt{B}$ in~\eqref{eq:auxiliarymatrix} is nonsingular for this particular choice of basis vectors. 
            The lemma holds for any basis of the left null space.
        \end{proof}

        \begin{theorem}
            \label{thm:eigenvaluecond}
            Consider an rMEP $\mep{\eig} \rvec = \vc{0}$ and eigenvalue $\sol{\eig}$.
            Let $\mt{B}$ be the auxiliary matrix defined in~\eqref{eq:auxiliarymatrix}.
            The relative eigenvalue condition number is given by
            \begin{equation}
                \label{eq:eigenvaluecond}
                \cond{\sol{\eig}} = \frac{\norm{\mt{B}^\inv}}{\norm{\sol{\eig}}} \sol{\gamma},
            \end{equation}
            in which $\sol{\gamma} = \norm{\bound_0} + \sum_{i = 1}^m \abs{\sol{\eigcomp[i]}} \norm{\bound_i}$ is similar to the definition of $\comp{\gamma}$. 
        \end{theorem}

        \begin{proof}
            From \cref{def:eigenvalueconditionnumber}, it follows that $\norm{\err{\eig}}/\norm{\sol{\eig}} \leq \cond{\sol{\eig}} \epsilon + \complexity{\epsilon^2}$.
            Let the coefficient matrices $\mat_i$ be submitted to perturbations of amplitude $\epsilon$.
            If we expand the first constraint and only keep the first-order terms, then we obtain
            \begin{equation}
                \mat_0 \err{\rvec} + \err{\mat_0} \sol{\rvec} + \sum_{i = 1}^m \left(\err{\eigcomp[i]} \mat_i \sol{\rvec} + \sol{\eigcomp[i]} \err{\mat_i} \sol{\rvec} + \sol{\eigcomp[i]} \mat_i \err{\rvec}\right) = \complexity{\epsilon^2}.
            \end{equation}
            In order to extract $\err{\eig}$, we premultiply this equation with $\mt{Y}^\herm$, which is an orthonormal basis matrix for the left null space of $\mep{\sol{\eig}}$. 
            Because $\mep{\sol{\eig}}$ is an $\left(\ell + m - 1\right) \times \ell$ matrix evaluated in a simple eigenvalue $\sol{\eig}$, its left null space has dimension equal to $m$ and $\mt{Y}$ consists of $m$ linearly independent rows $\lvec_i^\herm$. 
            Note that the non-trivial left eigenvector $\sol{\lvec}$ is an element of this vector space.
            The premultiplication results in 
            \begin{align}
                \mt{Y}^\herm \sum_{i = 1}^m \err{\eigcomp[i]} \mat_i \sol{\rvec} & = -\mt{Y}^\herm \err{\mat_0} \sol{\rvec} - \mt{Y}^\herm \sum_{i = 1}^m \sol{\eigcomp[i]} \err{\mat_i} \sol{\rvec} \\
                & \Downarrow \\
                \label{eq:deltaconstraint}
                \mt{B} 
                \begin{bmatrix}
                    \err{\eigcomp[1]} \\
                    \err{\eigcomp[2]} \\
                    \vdots \\
                    \err{\eigcomp[m]}
                \end{bmatrix}
                & = -
                \begin{bmatrix}
                    \lvec_1^\herm \left(\err{\mat_0} \sol{\rvec} + \sum_{i = 1}^m \sol{\eigcomp[i]} \err{\mat_i} \sol{\rvec}\right) \\
                    \lvec_2^\herm \left(\err{\mat_0} \sol{\rvec} + \sum_{i = 1}^m \sol{\eigcomp[i]} \err{\mat_i} \sol{\rvec}\right) \\
                    \vdots \\
                    \lvec_m^\herm \left(\err{\mat_0} \sol{\rvec} + \sum_{i = 1}^m \sol{\eigcomp[i]} \err{\mat_i}  \sol{\rvec}\right)
                \end{bmatrix}.
            \end{align}
            Because of \cref{lemma:genericity}, the auxiliary matrix $\mt{B}$ is invertible and $\norm{\err{\eig}} \leq \norm{\mt{B}^\inv} \sol{\gamma} \epsilon$.
            Consequently, \eqref{eq:eigenvaluecond} is an upper bound for $\cond{\sol{\eig}}$. 
            This upper bound can be attained for the perturbations
            \begin{gather}
                \err{\mat_0} = -\epsilon \norm{\bound_0} \sol{\lvec} \sol[\herm]{\rvec}, \\
                \err{\mat_i} = -\epsilon \sign\mleft(\sol{\eigcomp[i]}\mright) \norm{\bound_i} \sol{\lvec} \sol[\herm]{\rvec},
            \end{gather}
            which can be verified by substituting these perturbations into the feasibility conditions.
            It is straightforward to check that $\norm{\err{\mat_i}} = \epsilon \norm{\bound_i}$ for $i = 0, 1, \ldots, m$ (using the property that $\norm{\sol{\lvec} \sol[\herm]{\rvec}} = 1$).
            The remaining constraint, $\per{\sol{\eig} + \err \eig} \left( \sol{\rvec} + \err \rvec\right) = \vc{0}$, can be verified by substituting the perturbations into~\eqref{eq:deltaconstraint}, which equals
            \begin{equation}
                \mt{B} \err{\eig} = -
                \begin{bmatrix}
                    \lvec_1^\herm (-\epsilon \norm{\bound_0} \sol{\lvec} \sol[\herm]{\rvec} \rvec - \sum_{i = 1}^m  \eigcomp[i] \epsilon \sign (\sol{\eigcomp[i]}) \norm{\bound_i} \sol{\lvec} \sol[\herm]{\rvec} \rvec ) \\
                    \lvec_2^\herm (-\epsilon \norm{\bound_0} \sol{\lvec} \sol[\herm]{\rvec} \rvec - \sum_{i = 1}^m  \eigcomp[i] \epsilon \sign (\sol{\eigcomp[i]}) \norm{\bound_i} \sol{\lvec} \sol[\herm]{\rvec} \rvec ) \\
                    \vdots \\
                    \lvec_m^\herm (-\epsilon \norm{\bound_0} \sol{\lvec} \sol[\herm]{\rvec} \rvec - \sum_{i = 1}^m \eigcomp[i] \epsilon \sign (\sol{\eigcomp[i]}) \norm{\bound_i} \sol{\lvec} \sol[\herm]{\rvec} \rvec ) 
                \end{bmatrix}
                = \epsilon \sol{\gamma}
                \begin{bmatrix}
                    \lvec_1^\herm \sol{\lvec} \\
                    \lvec_2^\herm \sol{\lvec} \\
                    \vdots \\
                    \lvec_m^\herm \sol{\lvec}
                \end{bmatrix}.
            \end{equation}
            Since the left eigenvector $\sol{\lvec}$ is part of the left null space and orthogonal to the other basis vectors in $\mt{Y}$, the right factor corresponds to a unit vector $\vc{e}_j$.
            Left-multiplying by $\mt{B}^\inv$, taking the norm, and dividing by $\epsilon \norm{\sol{\eig}}$ gives~\eqref{eq:eigenvaluecond}.
        \end{proof}

        If we would consider the GEP, given by $\left(\mat_0 + \lambda \mat_1\right) \rvec = \vc{0}$, then the auxiliary matrix $\mt{B}$ from \cref{lemma:genericity} corresponds to the scalar $\sol[\herm]{\lvec} \mat_1 \sol{\rvec}$ and $\norm{\mt{B}^\inv} = 1 / (\sol[\herm]{\lvec} \mat_1 \sol{\rvec})$.
        The resulting condition number is then equal to 
        \begin{equation}
            \cond{\sol{\eig}} = \sol{\gamma} / (\sol[\herm]{\lvec} \mat_1 \sol{\rvec} \sol{\lambda}),                 
        \end{equation}
        which the same as in~\cite{fraysse1998note,higham1998structured} when considering normalized left and right eigenvectors.

        The absolute norm-wise eigenvalue condition number can be show to correspond to~\eqref{eq:eigenvaluecond} without $\norm{\sol{\eig}}$ in the denominator. 
        This looks exactly like a single equation equivalent of the absolute condition number derived in~\cite{hochstenbach2003backward} for coupled square multiparameter eigenvalue problem.

        \begin{remark}
            Notice that the condition number in~\eqref{eq:eigenvaluecond} is not valid for an eigenvalue with $\norm{\sol{\eig}} = 0$. 
            In such a case, one can measure the perturbation on the eigenvalue in an absolute way, by removing the norm from the denominator in~\eqref{eq:eigenvaluecond}.
        \end{remark}


        \begin{example}
            The condition numbers for the eigenvalues in \cref{ex:runningexample} are more or less the same: $\cond{\sol{\eig}_{(1)}} = \num{9.1899}$, $\cond{\sol{\eig}_{(2)}} = \num{6.9633}$, and $\cond{\sol{\eig}_{(3)}} = \num{11.2077}$.
        \end{example}

        \begin{remark}
            The condition number in~\eqref{eq:eigenvaluecond} is defined at the true eigenvalue $\sol{\eig}$, but in practice one has access only to a computed approximation $\comp{\eig}$.
            Each ingredient of~\eqref{eq:eigenvaluecond}, however, depends continuously on the evaluated simple eigenvalue.
            Consequently, $\cond{\comp{\eig}} \to \cond{\sol{\eig}}$ as $\comp{\eig} \to \sol{\eig}$, and thus $\cond{\comp{\eig}}$ is well-defined and accurate whenever the norm-wise backward eigenvalue error $\backward{\comp{\eig}}$ is small.
        \end{remark}


        \begin{example}
            \label{ex:otherexample}
            Another two-parameter eigenvalue problem is given by the following coefficient matrices:
            \begin{equation}
                \mat_0 = 
                \begin{bmatrix}
                    2 & 6 \\
                    4 & 5 \\
                    0 & 1 
                \end{bmatrix}, 
                \mat_1 = 
                \begin{bmatrix}
                    1 & 0 \\
                    0 & 1 \\
                    1 & 1 
                \end{bmatrix}, \eqand
                \mat_2 = 
                \begin{bmatrix}
                    4 & 2 \\
                    0 & 8 \\
                    1 & 1 
                \end{bmatrix}.
            \end{equation}
            Its three computed eigenvalue solutions, $\comp{\eig}_{(1)}$, $\comp{\eig}_{(2)}$, and $\comp{\eig}_{(3)}$, are given in \cref{tab:otherexample}, together with the norm-wise backward eigenvalue error, norm-wise relative eigenvalue condition numbers and angles of intersections of the secular equations.
            As further explained in \cref{prop:intersection}, a more tangent intersection of the secular equations leads to a higher condition number.
        \end{example}

        \begin{table}[t]
            \centering
            \caption{Numerical values of \cref{ex:otherexample}, motivating the results of \cref{prop:intersection}. Next to the computed eigenvalues $\comp{\eig} = (\comp{\eigcomp[1]}, \comp{\eigcomp[2]})$ with their norm-wise backward eigenvalue error $\backward{\comp{\eig}}$, we also give the norm-wise relative eigenvalue condition numbers $\cond{\comp{\eig}}$, the norm of the inverted auxiliary matrix $\mt{B}$ (which is essential in \cref{thm:eigenvaluecond}) and the mean intersection angle (in radians) of the secular equations at the computed eigenvalues. A higher eigenvalue condition number indicates that curves intersect at the eigenvalues that nearly parallel.}
            \label{tab:otherexample}
            \begin{tabular}{r|cccccc}
                \toprule
                & $\comp{\eigcomp[1]}$ & $\phantom{-}\comp{\eigcomp[2]}$ & $\backward{\comp{\eig}}$ & $\cond{\comp{\eig}}$ & $\norm{\mt{B}^\inv}$ & $\sum_{i = 1}^3 \theta_i / 3$ \\
                \midrule
                $\comp{\eig}_{(1)}$ & $3.6026$ & $-0.4183$ & $\num{7.1e-17}$ & $\num{37.9315}$ & $\num{7.3740}$ & $\qty{0.0977}{\radian}$ \\
                $\comp{\eig}_{(2)}$ & $1.3683$ & $\phantom{-}0.0552$ & $\num{1.7e-16}$ & $\num{62.8704}$ & $\num{7.3355}$ & $\qty{0.2077}{\radian}$ \\
                $\comp{\eig}_{(3)}$ & $0.9338$ & $-1.3750$ & $\num{1.9e-16}$ & $\num{11.9969}$ & $\num{0.9036}$ & $\qty{0.6283}{\radian}$ \\
                \bottomrule
            \end{tabular}
        \end{table}

        \begin{figure}
            \centering
            \tikzexternalenable
            \begin{tikzpicture}[trim axis left, trim axis right]
    \begin{axis}[
        figurestyle,
        xmin= 0,
        xmax= 4,
        ymin= -3,
        ymax= 1,
        xlabel = {$\lambda_1$},
        ylabel = {$\lambda_2$}
        ]

        \addplot +[blueline, no markers,
            raw gnuplot,
            name path = P
            ] gnuplot {
            set contour base;
            set cntrparam levels discrete 0.0001;
            unset surface;
            set view map;
            set isosamples 500;
            set samples 500;
            splot -x**2 - 9*x*y - x - 8*y**2 - y + 4;
        };
        \addlegendimage{blueline, no markers}
        \label{plot:otherexample:chi1}

        \addplot +[redline, dashed, no markers,
            raw gnuplot,
            name path = Q
            ] gnuplot {
            set contour base;
            set cntrparam levels discrete 0.003;
            unset surface;
            set view map;
            set isosamples 500;
            set samples 500;
            splot x**2 + 3*x*y + 2*y**2 + 2 - 3*x;
        };
        \addlegendimage{redline, dashed, no markers}
        \label{plot:otherexample:chi2}

        \addplot +[yellowline, dashdotted, no markers,
            raw gnuplot,
            name path = Q
            ] gnuplot {
            set contour base;
            set cntrparam levels discrete 0.0003;
            unset surface;
            set view map;
            set isosamples 500;
            set samples 500;
            splot x**2 + 12*x*y + 7*x + 32*y**2 + 28*y - 14;
        };
        \addlegendimage{yellowline, dashdotted, no markers}
        \label{plot:otherexample:chi3}

        \filldraw[black] (3.6026, -0.4183) circle (2pt) {};
        \filldraw[black] (1.3683, 0.0552) circle (2pt) {};
        \filldraw[black] (0.9338, -1.3750) circle (2pt) {};
    \end{axis}
\end{tikzpicture} 
            \tikzexternaldisable
            \caption{Real picture of the secular equations $\chi_1(\eig) = 0$ (\ref{plot:otherexample:chi1}), $\chi_2(\eig) = 0$ (\ref{plot:otherexample:chi2}), and $\chi_3(\eig) = 0$ (\ref{plot:otherexample:chi3}) of the two-parameter eigenvalue problem in \cref{ex:otherexample}. The intersections correspond to the three eigenvalues. The numerical values are given in \cref{tab:otherexample}. When the intersections are more tangent, a higher condition number is associated to that eigenvalue.}
            \label{fig:otherexample}
        \end{figure}
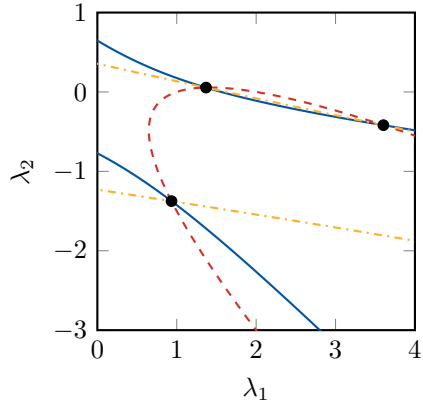

        \begin{proposition}
            \label{prop:intersection}
            Let $\sol{\eig} \in \Cset^2$ be an algebraically simple eigenvalue of the two-parameter eigenvalue problem $\mep{\eig}$ with coefficient matrices $\mat_i \in \Cset^{3 \times 2}$. 
            If $\chi_i(\eig)$, for $i = 1, 2, 3$, are the secular equations of $\mep{\eig}$ and $\mt{B}$ is the auxiliary matrix from \cref{lemma:genericity}, then the Jacobian matrix is equal to
            \begin{equation}
                \label{eq:intersection}
                \begin{bmatrix}
                    \dfrac{\partial \chi_1(\sol{\eig})}{\partial \eigcomp[1]} & \dfrac{\partial \chi_1(\sol{\eig})}{\partial \eigcomp[2]} \\
                    \dfrac{\partial \chi_2(\sol{\eig})}{\partial \eigcomp[1]} & \dfrac{\partial \chi_2(\sol{\eig})}{\partial \eigcomp[2]} \\
                    \dfrac{\partial \chi_3(\sol{\eig})}{\partial \eigcomp[1]} & \dfrac{\partial \chi_3(\sol{\eig})}{\partial \eigcomp[2]}
                \end{bmatrix}
                =
                \begin{bmatrix}
                    \alpha & 0 \\
                    0 & \beta \\
                    \gamma & \delta 
                \end{bmatrix}
                \vc{B},
            \end{equation}
            for some scalars $\alpha$, $\beta$, $\gamma$, and $\delta \in \Cset$.
            The Jacobian matrix is thus a fixed linear transformation of the auxiliary matrix $\mt{B}$.
        \end{proposition}

        \begin{proof}
            Consider the univariate square matrix pencil $\mcl{C}(t) = \left[\mep{t, \sol{\eigcomp[2]}}\right]_{\sigma_1}$, which corresponds to the first and second row of the matrix polynomial $\mep{\eig}$ with the second spectral parameter fixed at the solution $\sol{\eig}$.
            The next equality is a property of determinants,
            \begin{equation}
                \frac{\dd \mleft(\det \mcl{C}(t)\mright)}{\dd t} = \mathrm{trace}\mleft(\mathrm{adj} \, \mcl{C}(t) \frac{\dd \mleft(\mcl{C}(t)\mright)}{\dd t}\mright),
            \end{equation}
            from which it follows that the derivative of the first secular equation is equal to 
            \begin{equation}
                \frac{\dd \mleft(\chi_1(t, \sol{\eigcomp[2]})\mright)}{\dd t} = \mathrm{trace}\mleft(\mathrm{adj} \, \mcl{C}(t) \mt{P}_1 \mat_1\mright) = \mathrm{trace}\mleft(\alpha(t) \vc{x}(t) \vc{w}(t)^\herm \mt{P}_1 \mat_1\mright).
            \end{equation}
            The first equality follows from taking the derivative of $\mcl{C}(t) = \mt{P}_1 \mep{t, \sol{\eigcomp[2]}}$ with respect to $t$, where $\mt{P}_1$ performs the row selection, while the second equality uses the fact that the adjugate matrix of a rank-$(\ell - 1)$ matrix $\mt{M} \in \Cset^{\ell \times \ell}$ has rank one and in that case $\mathrm{adj}(\mt{M}) = \alpha \vc{x} \vc{w}^\herm$, where $\alpha$ is a scalar and $\vc{x}$ and $\vc{w}$ are null space vectors such that $\mt{M} \vc{x} = \vc{0}$ and $\vc{w}^\herm \mt{M} = \vc{0}$~\cite[p.~22]{horn2012matrix}.
            Evaluating $t$ in the first component of the solution $\sol{\eig}$ and reordering the elements within the trace operation results in
            \begin{equation}
                \frac{\partial \chi_1(\sol{\eig})}{\partial \eigcomp[1]} = \alpha \, \mathrm{trace} \mleft(\vc{w}(\eigcomp[1])^\herm \mt{P}_1 \mat_1 \vc{x}(\eigcomp[1])\mright) = \alpha \vc{w}(\sol{\eigcomp[1]})^\herm \mt{P}_1 \mat_1 \vc{x}(\sol{\eigcomp[1]}).
            \end{equation}
            Since $\mcl{C}(\sol{\eigcomp[1]})$ corresponds to a row selection of the evaluated matrix polynomial the right eigenvector lies in the right null space of that matrix and $\vc{x}(\sol{\eigcomp[1]}) = \sol{\rvec}$.
            Furthermore, we can combine the vector $\vc{w}(\sol{\eigcomp[1]})$ and row selection matrix into $\vc{w}_1 = \mt{P}_1^\herm \vc{w}(\sol{\eigcomp[1]})$, leading to
            \begin{equation}
                \frac{\partial \chi_1(\sol{\eig})}{\partial \eigcomp[1]} = \alpha \vc{w}_1^\herm \mat_1 \sol{\rvec}.
            \end{equation}

            If we repeat this approach for each secular equation (leading to the same scalar value) and each partial derivative, then we obtain
            \begin{equation}
                \begin{bmatrix}
                    \dfrac{\partial \chi_1(\sol{\eig})}{\partial \eigcomp[1]} & \dfrac{\partial \chi_1(\sol{\eig})}{\partial \eigcomp[2]} \\
                    \dfrac{\partial \chi_2(\sol{\eig})}{\partial \eigcomp[1]} & \dfrac{\partial \chi_2(\sol{\eig})}{\partial \eigcomp[2]} \\
                    \dfrac{\partial \chi_3(\sol{\eig})}{\partial \eigcomp[1]} & \dfrac{\partial \chi_3(\sol{\eig})}{\partial \eigcomp[2]}
                \end{bmatrix}
                =
                \begin{bmatrix}
                    \alpha & 0 & 0  \\
                    0 & \beta & 0 \\
                    0 & 0 & \epsilon 
                \end{bmatrix}
                \begin{bmatrix}
                    \vc{w}_1^\herm \vc{A}_1 \sol{\rvec} & \vc{w}_1^\herm \vc{A}_2 \sol{\rvec} \\
                    \vc{w}_2^\herm \vc{A}_1 \sol{\rvec} & \vc{w}_2^\herm \vc{A}_2 \sol{\rvec} \\
                    \vc{w}_3^\herm \vc{A}_1 \sol{\rvec} & \vc{w}_3^\herm \vc{A}_2 \sol{\rvec}
                 \end{bmatrix}.
            \end{equation}
            Since the dimension of the left null space of the matrix polynomial $\mep{\eig}$ is equal to two, there are only two linearly independent basis vectors for that left null space. 
            Thus, we can write $\vc{w}_1 = \lvec_1$, $\vc{w}_2 = \lvec_2$, and $\epsilon \vc{w}_3 = \gamma \lvec_1 + \delta \lvec_2$, leading to~\eqref{eq:intersection}.
        \end{proof}

        According to \cref{prop:intersection}, the intersection angle of the secular equations is proportional to the auxiliary matrix $\mt{B}$ in the eigenvalue condition number~\eqref{eq:eigenvaluecond}. 
        A higher eigenvalue condition number indicates that curves intersecting at the eigenvalue are nearly parallel, which can be observed in \cref{fig:otherexample}.

    \subsection{Eigenvector condition number}
        \label{sec:eigenvectorcondition}

        For completeness, we also consider the eigenvector condition number, for which we use a different projection of~\eqref{eq:fullmap} onto the eigenvector solutions, i.e., 
        \begin{equation}
            f_{\pi_{\rvec}} : \left\lbrace\mat_i\right\rbrace_{i = 0}^m \to \sol{\rvec}.
        \end{equation}
        Because an eigenvector corresponding to a simple eigenvalue is unique only up to scalar multiplication, it is important to normalize the eigenvectors in the following definition~\cite{higham1998structured}. 

        \begin{definition}
            Consider an rMEP $\mep{\eig} \rvec = \vc{0}$ and eigenpair $(\sol{\rvec}, \sol{\rvec})$.
            A \emph{norm-wise relative eigenvector condition number} can be defined by 
            \begin{equation}
                \begin{gathered}
                    \cond{\sol{\rvec}} = \lim_{\epsilon \to 0} \sup \frac{\norm{\err{\rvec}}}{\epsilon \norm{\sol{\rvec}}} \\
                    \begin{aligned}
                        \sub & \per{\sol{\eig} + \err{\eig}} \left(\sol{\rvec} + \err{\rvec}\right) = \vc{0}, \\
                        & \norm{\err{\mat_i}} \leq \epsilon \norm{\bound_i}, \quad i = 0, 1, \ldots, m, \\
                        & \vc{g}^\herm \sol{\rvec} = \vc{g}^\herm \left(\sol{\rvec} + \err{\rvec}\right) = 1.
                    \end{aligned}
                \end{gathered}
            \end{equation}    
            The linear normalization in the last constraint uses a constant vector $\vc{g} \in \Cset^\ell$.
        \end{definition}

        \begin{theorem}
            Consider an rMEP $\mep{\eig} \rvec = \vc{0}$ and eigenpair $(\sol{\eig}, \sol{\rvec})$.
            The norm-wise relative eigenvector condition number is given by 
            \begin{equation}
                \label{eq:eigenvectorcond}
                \cond{\sol{\rvec}} = \norm{\mt{V} \left(\mt{W}^\herm \mep{\sol{\eig}} \mt{V}\right)^\inv \mt{W}^\herm} \sol{\gamma},
            \end{equation}
            where $\sol{\gamma} = \norm{\bound_0} + \sum_{i = 1}^m \abs{\sol{\eigcomp[i]}} \norm{\bound_i}$ as before. 
            The two matrices $\mt{V} \in \Cset^{\ell \times (\ell - 1)}$ and $\mt{W} \in \Cset^{k \times (\ell - 1)}$ are specifically chosen full column rank matrices so that $\vc{g}^\herm \mt{V} = \vc{0}$ and $\mt{W}^\herm \mat_i \sol{\rvec} = \vc{0}$, for $i = 1, 2, \ldots, m$.
        \end{theorem}

        \begin{proof}
            The proof consists of three main steps: (i) projecting the perturbed problem onto a smaller subproblem, (ii) reformulating the eigenvector perturbation, and (iii) putting everything together in a norm-wise relative eigenvector condition number.

            (i) Consider the perturbed problem for sufficiently small perturbations, in which $\err{\rvec}$ is unique through the normalization with $\vc{g}$.
            We have that
            \begin{equation}
                \left(\mat_0 + \err{\mat_0} \right) \left( \sol{\rvec} + \err{\rvec} \right) + \sum_{i = 1}^m \left(\sol{\eigcomp[i]} + \err{\eigcomp[i]}\right) \left(\mat_i + \err{\mat_i}\right) \left(\sol{\rvec} + \err{\rvec}\right) = \vc{0},
            \end{equation}
            which is equal to  
            \begin{equation}
                \mep{\sol{\eig}} \err{\rvec} = - \err{\mep{\sol{\eig}}} \sol{\rvec} - \sum_{i = 1}^m \err{\eigcomp[i]} \mat_i \sol{\rvec}
            \end{equation}
            when we only consider first-order terms.
            The matrix $\mt{W}$ spans the orthogonal complement of $\spn\left\lbrace \mat_1 \sol{\rvec}, \mat_2 \sol{\rvec}, \ldots, \mat_m \sol{\rvec} \right\rbrace$.
            Premultiplying with $\mt{W}^\herm$ removes the last term and gives
            \begin{equation}
                \label{eq:projection}
                \mt{W}^\herm \mep{\sol{\eig}} \err{\rvec} = - \mt{W}^\herm \err{\mep{\sol{\eig}}} \sol{\rvec}.
            \end{equation}

            (ii) Define $\mt{M} = \begin{bmatrix} \sol{\rvec} & \mt{V} \end{bmatrix} \in \Cset^{\ell \times \ell}$ and $\mt{N} = \begin{bmatrix} \sol{\lvec} & \mt{W} \end{bmatrix} \in \Cset^{k \times l}$.
            To show that $\mt{M}$ is nonsingular, we consider $\mt{M} \vc{v} = \vc{0}$ and note that its first element $v_1$ is equal to zero because $0 = \vc{g}^\herm \mt{M} \vc{v} = \vc{e}_1^\tr \vc{v} = v_1$.
            Consequently, $\mt{M} \vc{v} = \mt{V} \vc{v}_2$, which is only zero when $\vc{v}_2$ (the remaining elements in $\vc{v}$) is zero, showing that $\mt{M}$ is nonsingular.
            Let $\err{\rvec} = \mt{M} \vc{x}$ for an unknown vector $\vc{x} \in \mathbb{C}^l$, then we can write
            \begin{align}
                \mt{W}^\herm \mep{\sol{\eig}} \err{\rvec} & = \mt{W}^\herm \mep{\sol{\eig}} \mt{M} \vc{x} \\
                & = \mt{W}^\herm \begin{bmatrix} \mep{\sol{\eig}} \sol{\rvec} & \mep{\sol{\eig}} \mt{V} \end{bmatrix} \begin{bmatrix} \vc{x}_1 \\ \vc{x}_2 \end{bmatrix} \\
                & = \mt{W}^\herm \mep{\sol{\eig}} \mt{V} \vc{x}_2.
            \end{align}
            Before continuing, we need to prove that $\mt{W}^\herm \mep{\sol{\eig}} \mt{V}$ is nonsingular.
            We use that
            \begin{align}
                \mt{N}^\herm \mep{\vc{\mu}} \mt{M} & = 
                \begin{bmatrix}
                    \lvec^\herm \\
                    \mt{W}^\herm
                \end{bmatrix} 
                \begin{bmatrix}
                    \sum_{i = 1}^m \left(\mu_i - \eigcomp[i]\right) \mat_i \sol{\rvec} & \mep{\vc{\mu}} \mt{V}
                \end{bmatrix} \\
                & = 
                \begin{bmatrix}
                    \lvec^\herm \sum_{i = 1}^m \left(\mu_i - \eigcomp[i]\right) \mat_i \sol{\rvec} & \lvec^\herm \mep{\vc{\mu}} \mt{V} \\
                    \mt{0} & \mt{W}^\herm \mep{\vc{\mu}} \mt{V}
                \end{bmatrix}.
            \end{align}
            Hence, $\det\mleft(\mt{N}^\herm \mep{\vc{\mu}}\mright) \det\mleft(\mt{M}\mright) = \lvec^\herm \sum_{i = 1}^m \left(\mu_i - \eigcomp[i]\right) \mat_i \sol{\rvec} \det\mleft(\mt{W}^\herm \mep{\vc{\mu}} \mt{V}\mright)$.
            Since $\mt{M}$ is nonsingular and $\mt{N}^\herm \mep{\vc{\mu}}$ only drops rank by one for $\vc{\mu} = \sol{\eig}$, it follows that $\det\mleft(\mt{W}^\herm \mep{\vc{\mu}} \mt{V}\mright)$ is nonzero for $\vc{\mu} = \sol{\eig}$, showing that it is nonsingular.
            The rank of the matrix product $\mt{N}^\herm \mep{\vc{\mu}}$ follows from the direct sum condition $\col\mleft(\mep{\vc{\mu}}\mright) \oplus \ker\mleft(\mt{N}^\herm\mright) = \Cset^k$ when $\vc{\mu} \neq \sol{\eig}$: 
            Since the rMEP is regular, $\mep{\vc{\mu}}$ has full column rank for $\vc{\mu} \notin \spec{\mep{\eig}}$, so $\col\mleft(\mep{\vc{\mu}}\mright)$ has dimension $\ell$ and $\ker\mleft(\mt{N}^\herm\mright)$ has dimension $k - \ell$ by construction of $\mt{N}$. 
            Their dimensions sum to $k$, and they are complementary by the choice of $\mt{N}$, so $\mt{N}^\herm$ is injective on $\col\mleft(\mep{\vc{\mu}}\mright)$ and $\mt{N}^\herm \mep{\vc{\mu}}$ has full column rank $\ell$.

            (iii) Combining the projected equation~\eqref{eq:projection} from (i) with the decomposition from (ii), we have that
            \begin{equation}
                \mt{W}^\herm \mep{\sol{\eig}} \mt{V} \vc{x}_2 = -\mt{W}^\herm \err{\mep{\sol{\eig}}} \sol{\rvec}.
            \end{equation}
            Since $\mt{W}^\herm \mep{\sol{\eig}} \mt{V}$ is nonsingular, we can solve for $\vc{x}_2$:
            \begin{equation}
                \vc{x}_2 = -\left(\mt{W}^\herm \mep{\sol{\eig}} \mt{V}\right)^\inv \mt{W}^\herm \err{\mep{\sol{\eig}}} \sol{\rvec}.
            \end{equation}
            Substituting back via $\err{\rvec} = \mt{V} \vc{x}_2$, the norm of the eigenvector perturbation satisfies
            \begin{align}
                \norm{\err{\rvec}} & = \norm{\mt{V} \left(\mt{W}^\herm \mep{\sol{\eig}} \mt{V}\right)^\inv \mt{W}^\herm \err{\mep{\sol{\eig}}} \sol{\rvec}} \\
                & \leq \norm{\mt{V} \left(\mt{W}^\herm \mep{\sol{\eig}} \mt{V}\right)^\inv \mt{W}^\herm} \norm{\sol{\rvec}} \sol{\gamma} \epsilon,
            \end{align}
            where the inequality follows from the backward error analysis in \cref{sec:error}.
            Consequently, it follows that $\cond{\sol{\rvec}}$ is given by the expression in~\eqref{eq:eigenvectorcond}.
        \end{proof}

\section{Multiparameter pseudospectrum}
    \label{sec:pseudospectrum}

    Pseudospectra help to analyze rMEPs by characterizing the sensitivity of the approximate eigenvalues to perturbations in the coefficient matrices.
    After discussing definitions and properties (\cref{sec:pseudotheory}), we show how to visualize the pseudospectrum of an rMEP (\cref{sec:pseudovisualization}) and dive into its practical computation (\cref{sec:pseudocomputation}).

    \subsection{Different definitions and some useful properties}
        \label{sec:pseudotheory}

        The $\epsilon$-pseudospectrum of an rMEP $\mep{\eig} \rvec = \vc{0}$ is a set of tuples $\pseudo{\mep{\eig}}$ given by the following definition.

        \begin{definition}
            \label{def:pseudospectrum}
            For an rMEP $\mep{\eig}$ with coefficient matrices $\mat_0, \mat_1, \ldots, \mat_m$, the set
            \begin{equation}
                \pseudo{\mep{\eig}} = \big\lbrace\comp{\eig} \in \Cset^m : \exists \comp{\rvec} \neq \vc{0}, \per{\comp{\eig}} \comp{\rvec} = \vc{0} \text{~with~} \norm{\err{\mat_i}} \leq \epsilon \norm{\bound_i}\big\rbrace
            \end{equation}
            is the \emph{$\epsilon$-pseudospectrum} of $\mep{\eig}$.
            The matrix polynomial $\per{\eig}$ corresponds to $\mep{\eig}$ with perturbated coefficient matrices $\mat_i + \err \mat_i$.
        \end{definition}

        \begin{theorem}
            \label{thm:equivalences}
            For a linear rMEP $\mep{\eig}$, the definitions of the sets
            \begin{enumerate}[label=(\roman*)]
                \item $\pseudo{\mep{\eig}} = \big\lbrace\comp{\eig} \in \Cset^m : \exists \comp{\rvec} \, \text{with} \, \norm{\comp{\rvec}} = 1, \norm{\mep{\comp{\eig}} \comp{\rvec}} \leq \epsilon \comp{\gamma}\big\rbrace$,
                \item $\pseudo{\mep{\eig}} = \big\lbrace\comp{\eig} \in \Cset^m : \sigma_\mathrm{min} \mleft(\mep{\comp{\eig}}\mright) \leq \epsilon \comp{\gamma}\big\rbrace$, and
                \item $\pseudo{\mep{\eig}} = \big\lbrace\comp{\eig} \in \Cset^m : \Vert\mep{\comp{\eig}}^\pinv\Vert \geq (\epsilon \comp{\gamma})^\inv\big\rbrace$,
            \end{enumerate}
            are equivalent with \cref{def:pseudospectrum}.
        \end{theorem}

        \begin{proof}
            We start by showing that \cref{def:pseudospectrum} is equivalent with (i). 
            \begin{itemize}
                \item Suppose that $\comp{\eig} \in \pseudo{\mep{\eig}}$, then $\mep{\comp{\eig}} \comp{\rvec} = - \err{\mat_0} \comp{\rvec} - \sum_{i = 1}^m \eigcomp[i] \err{\mat_i} \comp{\rvec}$, implying that $\norm{\mep{\comp{\eig}} \comp{\rvec}} = \norm{\err{\mat_0} \comp{\rvec} + \sum_{i = 1}^m \eigcomp[i] \err{\mat_i} \comp{\rvec}} \leq \norm{\err{\mat_0} + \sum_{i = 1}^m \eigcomp[i] \err{\mat_i}}$.
                In the proof of \cref{thm:backwarderror}, we have already shown that this last expression is smaller or equal to $\epsilon \comp{\gamma}$, proving that $\comp{\eig}$ is also part of $\pseudo{\mep{\eig}}$ as defined in (i), which proves the $\Rightarrow$ direction.
                \item Let $\comp{\eig} \in \pseudo{\mep{\eig}}$ now be part of the set defined in (i) for the $\Leftarrow$ direction.
                Consider the residual $\res = \mep{\comp{\eig}} \comp{\rvec}$. 
                There exists a matrix $\mt{H} \in \Cset^{k \times \ell}$ with $\norm{\mt{H}} = 1$ so that $\mt{H} \comp{\rvec} = - \res/\norm{\res}$.
                If $\bound_0 = \norm{\res} \mt{H}$, then $\left(\mep{\comp{\eig}} + \bound_0\right) \comp{\rvec} = \res + \norm{\res} \mt{H} \comp{\rvec} = \vc{0}$.
                Consequently, $\norm{\bound_0} = \norm{\mep{\comp{\eig}} \comp{\rvec}} \leq \epsilon \comp{\gamma}$.
                We have shown that a point that belongs to the set in (i) also satisfies the matrix equation $\left(\mep{\comp{\eig}} + \bound_0\right) \comp{\rvec} = \vc{0}$.
                It is now easy to see $\bound_0$ as a perturbation to $\mat_0$, resulting in \cref{def:pseudospectrum}.
                To make it more complete, we can distribute $\bound_0$ over the different coefficient matrices, 
                \begin{gather}
                    \err{\mat_0} = \frac{\norm{\mat_0} \bound_0}{\comp{\gamma}}, \\
                    \err{\mat_i} = \frac{\sign\mleft(\eigcomp[i]\mright) \norm{\mat_i} \bound_0}{\comp{\gamma}}, \quad i = 1, 2, \ldots, m,
                \end{gather}
                proving the other direction.
            \end{itemize}
            Next, we show the equivalence of (i), (ii), and (iii).
            \begin{itemize}
                \item (i) $\Rightarrow$ (ii): 
                Assume that $\comp{\eig} \in \pseudo{\mep{\eig}}$ by (i).
                Then there exists a vector $\comp{\rvec}$ for which $\norm{\comp} = 1$ so that $\norm{\mep{\comp{\eig}} \comp{\rvec}} \leq \epsilon \comp{\gamma}$, which is equivalent with the definition in (ii) via $\sigma_\mathrm{min}\mleft(\mep{\comp{\eig}}\mright) = \min_{\rvec} \norm{\mep{\comp{\eig}} \rvec} \leq \norm{\mep{\comp{\eig}} \comp{\rvec}} \leq \epsilon \comp{\gamma}$.
                \item (ii) $\Rightarrow$ (iii): 
                We can write $\Vert\mep{\comp{\eig}}^\pinv\Vert = \sigma^{-1}_{\mathrm{min}}(\mep{\comp{\eig}})$. 
                Subsequently, we know from (ii) that $\comp{\eig} \in \Cset^m: \sigma_\mathrm{min}\mleft(\mep{\comp{\eig}}\mright) \leq \epsilon \comp{\gamma}$, which implies $\sigma_\mathrm{min}^\inv\mleft(\mep{\comp{\eig}}\mright) \geq (\epsilon \comp{\gamma})^\inv$, thus $\Vert\mep{\comp{\eig}}^\pinv\Vert \geq (\epsilon \comp{\gamma})^\inv$.
                \item (iii) $\Rightarrow$ (i):
                Since $\Vert\mep{\comp{\eig}}^\pinv\Vert \geq (\epsilon \comp{\gamma})^\inv$, we can write that $\sigma_\mathrm{min}^\inv\mleft(\mep{\comp{\eig}}\mright) \geq (\epsilon \comp{\gamma})^\inv$, or $\sigma_\mathrm{min}\mleft(\mep{\comp{\eig}}\mright) \leq \epsilon \comp{\gamma}$. 
                We know that $\min_{\norm{x} = 1} \norm{\mep{\comp{\eig}} \vc{x}} = \sigma_\mathrm{min}\mleft(\mep{\comp{\eig}}\mright)$, noting that the minimum norm is achieved at some $\comp{\rvec}$, then we have that $\Vert\mep{\comp{\eig}} \comp{\rvec}\Vert \leq \epsilon \comp{\gamma}$.
            \end{itemize}
            Combining all equivalences shows that all expressions are equivalent.
        \end{proof}

        \begin{corollary}
            \label{cor:singularvalues}
            We can re-write the second definition in \cref{thm:equivalences} in terms of the norm-wise backward eigenvalue error, so that $\pseudo{\mep{\eig}}$ corresponds to the set of all approximate eigenvalues whose norm-wise backward eigenvalue error $\backward{\comp{\eig}} \leq \epsilon$.
        \end{corollary}

        \begin{example}
            In \cref{fig:pseudoexamples}, we visualize the pseudospectra, $\pseudo{\mep{\eig}}$, of the two-parameter eigenvalue problems in both \cref{ex:runningexample} and \cref{ex:otherexample}.
        \end{example}

        \begin{figure}
            \centering
            \begin{subfigure}{0.45\textwidth}
                \centering
                \tikzexternalenable
                \input{figures/runningexamplecontour.tikz}
                \tikzexternaldisable
                \caption{$\pseudo{\mep{\eig}}$ for \cref{ex:runningexample}}
                \label{fig:runningexamplecontour}
            \end{subfigure}
            \begin{subfigure}{0.45\textwidth}
                \centering
                \includegraphics{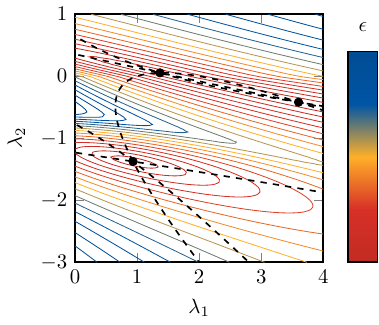}
                \caption{$\pseudo{\mep{\eig}}$ for \cref{ex:otherexample}}
                \label{fig:otherexamplecontour}
            \end{subfigure}
            \caption{Visualization of the pseudospectra, $\pseudo{\mep{\eig}}$, of the two-parameter eigenvalue problems in both \cref{ex:runningexample} and \cref{ex:otherexample}. The secular equations are also shown (\ref{plot:secularequations}).}
            \label{fig:pseudoexamples}
        \end{figure} 

%
    \subsection{Visualization of the pseudosectrum}
        \label{sec:pseudovisualization}

        Visualizing the pseudospectrum of a single-parameter eigenvalue problem in the complex plane is straightforward: the two axes represent the real and imaginary parts of the eigenvalue, and the plane is partitioned by contour lines corresponding to the minimum singular value of the matrix pencil evaluated at a given complex point. 
        A similar complete representation is not tractable for more than two spectral parameters. 
        Hence, we focus here on a special class of rMEPs for which all eigenvalues are real, allowing a visualization of the pseudospectrum of two-parameter eigenvalue problems in the two-dimensional plane.

        \begin{definition}
            \label{def:rightdefinite}
            An rMEP $\mep{\eig} \rvec = \vc{0}$ is \emph{(nested) right definite} if there exists a nonsingular, right definite coupled square multiparameter eigenvalue problem of the form
            \begin{equation}
                \mt{P}_i^\tr \mep{\eig} \rvec = \vc{0}, \quad i = 1, 2, \ldots, m
            \end{equation}
            where $\mt{P}_i \in \Rset^{\ell \times k}$ are row selection matrices that select $\ell$ different rows.
        \end{definition}

        Recall that a nonsingular, right definite coupled square multiparameter eigenvalue $\smep[i]{\eig} \rvec_i = \mt{V}_{i0} \rvec_i + \sum_{j = 1}^m \eigcomp[j] \mt{V}_{ij} \rvec_i = \vc{0}$, for $i = 1, 2, \ldots, m$, is a problem that has a positive definite delta matrix $\mt{\Delta}_0 = \sum_{\sigma \in S_m} \sign(\sigma) \mt{V}_{1\sigma_1} \otimes \mt{V}_{2\sigma_2} \otimes \cdots \otimes \mt{V}_{m\sigma_m}$, where $\otimes$ denotes the Kronecker operator and $\sigma \in S_m$ corresponds to the permutations of a set of $m$ numbers~\cite{atkinson1972multiparameter}.
        All eigenvalues of such nonsingular, right definite problem are real~\cite{hochstenbach2003backward}.

        \begin{lemma} 
            \label{lemma:smepintersection}
            The spectrum $\spec{\mep{\eig}}$ of an rMEP $\mep{\eig} \rvec = \vc{0}$ is given by, with $L = \binom{k}{\ell}$,
            \begin{equation}
                \spec{\mep{\eig}} = \bigcap_{i = 1}^L \spec{\smep{\eig}},
            \end{equation}  
            where $\smep{\eig} = \mt{P}_i^\tr \mep{\eig} \in \mathbb{R}^{\ell \times \ell}$, $i = 1, 2, \ldots, L$ are square submatrix polynomials for which each $\mt{P}_i$ corresponds to a distinct selection of $\ell$ rows.
        \end{lemma}

        \begin{proof}
            If $\sol{\eig} \in \spec{\mep{\eig}}$, then $\mep{\eig}$ drops rank at $\sol{\eig}$ (\cref{cor:rankdefinition}) and there exists a nonzero vector $\sol{\rvec} \in \Cset^\ell_{\vc{0}}$ such that $\mep{\sol{\eig}} \sol{\rvec} = \vc{0}$. 
            Subsequently, $\mt{P}_i^\tr \mep{\sol{\eig}} \sol{\rvec} = \smep{\sol{\eig}} \sol{\rvec} = \vc{0}$.  Therefore, $\spec{\mep{\eig}} \subseteq \bigcap_{i = 1}^L \spec{\smep{\eig}}$.
            Conversely, let $\sol{\eig} \in \bigcap_{i = 1}^L \spec{\smep{\eig}}$.
            Every $\ell \times \ell$ square polynomial submatrix $\smep{\sol{\eig}}$ is singular.
            Hence, all $\ell \times \ell$ maximal minors of $\mep{\sol{\eig}}$ vanish and $\rank(\mep{\sol{\eig}}) < \ell$. 
            There exists a nonzero vector $\sol{\rvec} \in \Cset^\ell_{\vc{0}}$ such that $\mep{\sol{\eig}} \sol{\rvec} = \vc{0}$. 
            So, $\sol{\eig} \in \spec{\mep{\eig}}$ and $\bigcap_{i = 1}^L \spec{\smep{\eig}} \subseteq \spec{\mep{\eig}}$. 
            Combining both inclusions proves the result. 
        \end{proof}

        \begin{remark}
            In \cref{lemma:smepintersection}, the intersection runs over all $L = \binom{k}{\ell}$ possible square submatrix polynomials.
            By contrast, the definition of the (nested) right definite rMEP (\cref{def:rightdefinite}) requires only $m$ specific row selections---one per spectral parameter---that together form a right definite coupled square multiparameter eigenvalue problem. 
            One of the repercussions is that the spectra in \cref{lemma:smepintersection} agree, while the spectrum is a subset of the coupled square multiparameter eigenvalue problem in \cref{def:rightdefinite}.
        \end{remark}

        \Cref{lemma:smepintersection} shows that the spectrum of the rMEP is related to the spectra of the associated square polynomial submatrices.
        We now extend this idea to the pseudosectrum.
        \begin{lemma} 
            \label{lemma:pseudospectrumintersection}
            The $\epsilon$-pseudospectrum $\pseudo{\mep{\eig}}$ of an rMEP $\mep{\eig} \rvec = \vc{0}$ satisfies, with $L = \binom{k}{\ell}$,
            \begin{equation}
                \pseudo{\mep{\eig}} \subseteq \bigcap_{i = 1}^L \pseudo{\smep{\eig}},
            \end{equation}
            where $\smep{\eig} = \mt{P}_i^\tr \mep{\eig} \in \mathbb{R}^{\ell \times \ell}$, $i = 1, 2, \ldots, L$ are square submatrix polynomials for which each $\mt{P}_i$ corresponds to a distinct selection of $\ell$ rows.
        \end{lemma}

        \begin{proof}
            Let $\sol{\eig} \in \pseudo{\mep{\eig}}$, or equivalently by \cref{thm:equivalences},
            $\sigma_\mathrm{min}\mleft(\mep{\sol{\eig}}\mright) \leq \epsilon \gamma$. 
            Now fix any $i \in \{1, 2, \ldots, L\}$. 
            Since $\smep{\sol{\eig}} = \mt{P}_i^\tr \mep{\sol{\eig}}$ and $\mt{P}_i^\tr$ is a row-selection matrix with $\norm{\mt{P}_i^\tr} = 1$, we have that for every $\rvec \in \Cset^\ell$ with $\norm{\rvec} = 1$ the norm $\norm{\smep{\sol{\eig}} \rvec} = \norm{\mt{P}_i^\tr \mep{\sol{\eig}} \rvec} \leq \norm{\mep{\sol{\eig}} \rvec}$.
            Taking the minimum over all unit vectors $\rvec$ gives $\sigma_\mathrm{min}\mleft(\smep{\sol{\eig}}\mright) \leq \sigma_\mathrm{min}\mleft(\mep{\sol{\eig}}\mright)$.
            Therefore, $\sigma_\mathrm{min}\mleft(\smep{\sol{\eig}}\mright) \leq \epsilon \gamma$ and, again by \cref{thm:equivalences}, $\sol{\eig} \in \pseudo{\smep{\eig}}$.
            Since $i$ was arbitrary, it follows that $\sol{\eig} \in \bigcap_{i = 1}^L \pseudo{\smep{\eig}}$ and $\pseudo{\mep{\eig}} \subseteq \bigcap_{i = 1}^L \pseudo{\smep{\eig}}$.
        \end{proof}       

        \begin{remark}
            The converse inclusion in \cref{lemma:pseudospectrumintersection} does not hold in general.
            Since $\smep{\sol{\eig}} = \mt{P}_i^\tr \mep{\sol{\eig}}$ with $\mt{P}_i^\tr$ a row-selection matrix, it follows that $\sigma_\mathrm{min}\mleft(\smep{\sol{\eig}}\mright) \leq \sigma_{\min}\mleft(\mep{\sol{\eig}}\mright)$ for all $i$: deleting rows cannot increase the smallest singular value, or equivalently, adding rows may increase it.
            Consequently, there may exist a point $\sol{\eig} \in \Cset^m$ such that $\sigma_\mathrm{min}\mleft(\smep{\sol{\eig}}\mright) \leq \epsilon \gamma$ for all $i = 1, 2, \ldots, L$, while $\sigma_\mathrm{min}\mleft(\mep{\sol{\eig}}\mright) > \epsilon \gamma$.
            In that case, $\sol{\eig} \in \bigcap_{i=1}^L \pseudo{\smep{\eig}}$ but $\sol{\eig} \notin \pseudo{\mep{\eig}}$, so the reverse inclusion fails.
        \end{remark}

        \begin{theorem}
            The eigenvalues of a (nested) right definite rMEP are real. 
            If the coefficient matrices are all real, then there also exist corresponding real eigenvectors.
        \end{theorem}

        \begin{proof}
            It follows from \cref{lemma:smepintersection} that the rMEP spectrum is the intersection of all submatrix spectra, which is a subset of the coupled square one (\cref{def:rightdefinite} gives $m$ of the $L$ submatrices), and the coupled square spectrum is real by~\cite{hochstenbach2003backward}, so the intersection is real.
            All eigenvalues of the (nested) right definite rMEP are thus real. 

            If the coefficient matrices of the rMEP are real, then $\mep{\sol{\eig}}$ is a real matrix and we can consider real eigenvectors. 
        \end{proof}

        \begin{remark}
            The requirement on the rMEP of being (nested) right definite is very restrictive. 
            We cannot find a general approach to construct rMEPs with $\ell > 2$ and/or $m > 2$ for which a corresponding right definite coupled square multiparameter eigenvalue problem exists. 
            The proof above is non-constructive in that regard.
            However, right definite two-parameter eigenvalue problems with $3 \times 2$ coefficient matrices do exist.  
            For example, when the coefficient matrices have a Hankel structure, which ensures that at least two submatrices are Hermitian, the rMEP is (nested) right definite. 
        \end{remark}

        \begin{example} 
            \label{ex:rightdefinite}

            Consider the right definite two-parameter eigenvalue problem
            \begin{equation}
                \label{eq:rightdefinite}
                \mep{\eig} \rvec = \left(\mat_0 + \eigcomp[1] \mat_1 + \eigcomp[2] \mat_2\right) \rvec = \vc{0},
            \end{equation}
            with coefficient matrices
            \begin{gather}
                \mat_0 = \begin{bmatrix}
                    0.0260 & 0.4380 \\
                    0.4380 & 0.6542 \\
                    0.6542 & 1.4192
                \end{bmatrix}, \mat_1 = \begin{bmatrix}
                    -1.6324 & \phantom{-}0.1128\\
                    \phantom{-}0.1128  & -0.4401\\
                    -0.4401 & -0.6968
                \end{bmatrix}, \eqand \\
                \mat_2 = \begin{bmatrix}
                    \phantom{-}3.3280 & -0.3346\\
                    -0.3346 & \phantom{-}0.6774\\
                    \phantom{-}0.6774 & -0.0754
                \end{bmatrix}.
            \end{gather}
            The secular equations of~\eqref{eq:rightdefinite} are obtained by taking the determinants of the $2 \times 2$ subpencils $\smep{\eig} = \mt{P}_i^\tr \mep{\eig}$, in which the selection matrices $\mt{P}_i$ select rows $\left\lbrace1, 2\right\rbrace$, $\left\lbrace2, 3\right\rbrace$, and $\left\lbrace1, 3\right\rbrace$, resulting in  
            \begin{align}
                \chi_{1}(\eig) & = \num{0.7057} \eigcomp[1]^2 - \num{2.4950} \eigcomp[1] \eigcomp[2] - \num{1.1782} \eigcomp[1] + \num{2.1424} \eigcomp[2]^2 + \num{2.4879} \eigcomp[2] - \num{0.1748} = 0, \\
                \chi_{2}(\eig) & = -\num{0.2723} \eigcomp[1]^2 + \num{0.8209} \eigcomp[1] \eigcomp[2] + \num{0.4307} \eigcomp[1] - \num{0.4336} \eigcomp[2]^2 - \num{1.3942} \eigcomp[2] + \num{0.1936} = 0, \\
                \chi_{3}(\eig) & = \num{1.1871} \eigcomp[1]^2 - \num{2.4195} \eigcomp[1] \eigcomp[2] - \num{2.2158} \eigcomp[1] - \num{0.0243} \eigcomp[2]^2 + \num{4.6433} \eigcomp[2] - \num{0.2496} = 0.
            \end{align}
            The $\epsilon$-pseudospectra and secular equations of the square polynomial submatrices $\smep{\eig}$, together with those of the rectangular matrix polynomial $\mep{\eig}$, are shown in \cref{fig:pseudospectra}.
            The $\epsilon$-pseudospectrum of the rMEP is a combination of the $\epsilon$-pseudospectra of the square submatrix polynomials, while the eigenvalues are situated at the intersections of the secular equations. 
            For the eigenvalue $\comp{\eig}_{(1)} = (\num{8.9620}, \num{4.3879})$, the condition number is $\num{89.5856}$; for $\comp{\eig}_{(2)} = (\num{0.4780}, \num{0.2982})$ it is $\num{48.4558}$; and for $\comp{\eig}_{(3)} = (\num{1.9804}, \num{0.1180})$ it is $\num{5.8383}$.
        \end{example}

        \begin{figure}[t]
            \centering
            \begin{subfigure}{0.45\textwidth}
                \centering
                \includegraphics{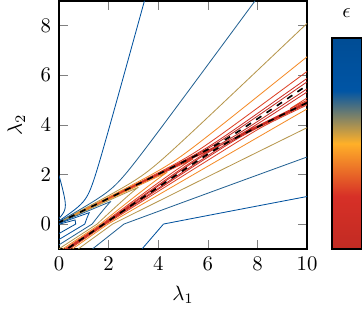}
                \caption{$\Lambda_{\epsilon}(\mcl{B}_1(\eig))$}
                \label{fig:pseudospectra:firstsubpicture}
            \end{subfigure}
            \begin{subfigure}{0.45\textwidth}
                \centering
                \includegraphics{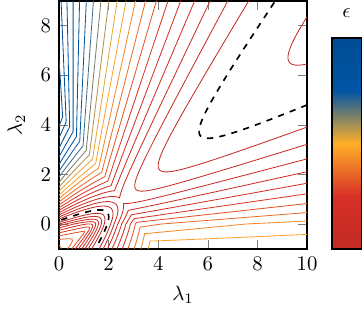}
                \caption{$\Lambda_{\epsilon}(\mcl{B}_2(\eig))$}
                \label{fig:pseudospectra:secondsubpicture}
            \end{subfigure}
            \begin{subfigure}{0.45\textwidth}
                \centering
                \bigskip
                \includegraphics{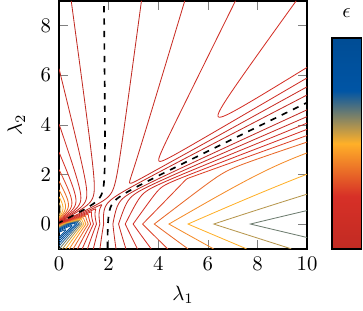}
                \caption{$\Lambda_{\epsilon}(\mcl{B}_3(\eig))$}
                \label{fig:pseudospectra:thirdsubpicture}
            \end{subfigure}
            \begin{subfigure}{0.45\textwidth}
                \centering
                \bigskip
                \includegraphics{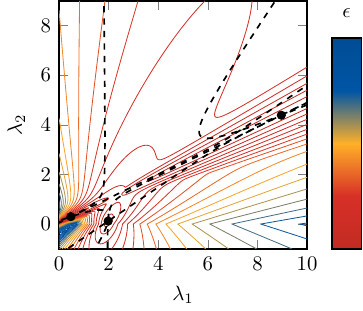}
                \caption{$\Lambda_{\epsilon}(\mep{\eig})$}
                \label{fig:pseudospectra:fullpicture}
            \end{subfigure}
            \caption{Secular equations and $\epsilon$-pseudospectra of the different matrix polynomials in \cref{ex:rightdefinite}. The spectrum of the rMEP $\mep{\eig}$ is obtained as the intersection of the spectra of the square submatrix polynomials $\smep{\eig}$, for $i = 1, 2, 3$. A similar relation holds for the $\epsilon$-pseudospectra. The secular curves (\ref{plot:secularequations}) for $\comp{\eig}_{(1)} = (\num{8.9620}, \num{4.3879})$ and $\comp{\eig}_{(2)} = (\num{0.4780}, \num{0.2982})$ are locally almost parallel, while this is not the case for $\comp{\eig}_{(3)} = (\num{1.9804}, \num{0.1180})$, resulting in a significantly lower condition number; another illustration of \cref{prop:intersection}.}
            \label{fig:pseudospectra}
        \end{figure}

    \subsection{Numerical computation of the pseudospectrum}
        \label{sec:pseudocomputation}

        Computing the $\epsilon$-pseudospectrum typically resorts on the second characterization in \cref{thm:equivalences}.
        For a given tuple $\comp{\eig} \in \Cset^m$, one evaluates the smallest singular value of the rectangular matrix $\mep{\comp{\eig}}$. 
        Then, $\comp{\eig} \in \pseudo{\mep{\eig}}$ if and only if $\sigma_\mathrm{min}\mleft(\mep{\comp{\eig}}\mright) \leq \epsilon \comp{\gamma}$.
        The numerical computation of the $\epsilon$-pseudospectrum reduces to evaluating
        $\sigma_\mathrm{min}\mleft(\mep{\eig}\mright)$ over a prescribed grid in the spectral parameter space $\Cset^m$.

        For a generalized rectangular matrix pencil with a single parameter, the standard approach is to first transform the matrix pencil by unitary equivalences into a trapezoidal form~\cite{trefethen2005spectra,wright2002pseudospectra}. 
        Since unitary transformations preserve the singular values, the preprocessing does not alter the pseudospectrum. 
        Its purpose is purely computational: 
        Once the trapezoidal form has been obtained, the subsequent QR factorization is cheaper and the minimum singular value can be computed more efficiently. 
        Instead of evaluating the original rectangular matrix pencil directly at each grid point, one evaluates an equivalent matrix pencil with a more favorable structure. 

        The proposed approach to numerically compute the $\epsilon$-pseudospectra of rMEPs is based on the computation of the pseudospectrum for a one-parameter matrix pencil. 
        In what follows, we describe the required preprocessing steps and the corresponding computational complexity of the methods in~\cite{trefethen2005spectra,wright2002pseudospectra} when used for multiparameter pseudospectrum computations.  
        The construction of this trapezoidal form depends on the relation between $k$ and $\ell$: it can be slightly tall ($k < 2 \ell$) or very tall ($k \geq 2 \ell$).            
        In both cases, the essential point is that the $\epsilon$-pseudospectrum is computed not through a simultaneous reduction of all coefficient matrices, but rather by fixing all but one parameter and treating the resulting slice as an one-parameter rectangular pencil. 
        The full multiparameter pseudospectrum is then obtained by repeating this computation over all slices of the parameter grid.\footnote{The function \function{mppseudo} provides an implementation of such an algorithm. Although different methods can be used to compute the one-parameter problems, we resort on \toolbox{EigTool}~\cite{wright2002eigtool} for that.}


        \subsubsection{Slightly tall case}

            In the slightly tall case, $k < 2 \ell$, the required structure in the one-parameter setting is obtained by means of a QZ decomposition of the lower $\ell \times \ell$ square block of the two rectangular matrices $\mat_0$ and $\mat_1$.  

            No analogous decomposition exists for more than two matrices simultaneously, so, for an $m$-parameter eigenvalue problem, the $\epsilon$-pseudospectrum must be computed slice-wise.
            More precisely, let each parameter $\eigcomp$ vary over a rectangular complex grid
            $G_i \subset \Cset$ of the form 
            \begin{equation}
                G_i = \left\lbrace x + \cmp y : x \in (\alpha_i, \beta_i), y \in (\eta_i, \zeta_i), \alpha_i,\beta_i, \eta_i, \zeta_i \in \Rset\right\rbrace,
            \end{equation}
            discretized by points in both the real and imaginary direction. 
            To compute the pseudospectrum over the full parameter grid $G_1 \times G_2 \times \cdots \times G_m$, first fix $(\comp{\eigcomp[1]}, \comp{\eigcomp[2]}, \ldots, \comp{\eigcomp[m - 1]}) \in G_1 \times G_2 \times \cdots \times G_{m-1}$.
            For each such choice, consider the sliced linear matrix pencil $(\comp{\eigcomp[1]}, \comp{\eigcomp[2]}, \ldots, \comp{\eigcomp[m - 1]}, \eigcomp[m])$ as a one-parameter rectangular matrix pencil in the remaining variable $\eigcomp[m]$.
            The one-parameter pseudospectrum algorithm for the generalized pencil with $k < 2 \ell$ is then applied to compute the pseudospectrum over $\lambda_m \in G_m$.
            Repeating this procedure for all slices, $(\eigcomp[1], \eigcomp[2], \ldots, \eigcomp[m - 1]) \in G_1 \times G_2 \times \cdots \times G_{m-1}$, yields the pseudospectrum on the full multiparameter grid $G_1 \times G_2 \times \cdots \times G_m$.

            Writing $N_\mathrm{slice} \coloneq \prod_{i = 1}^{m - 1} \abs{G_i}$ for the number of slices and $N_m \coloneq \abs{G_m}$ for the number of grid points in the remaining parameter, the computational complexity for the ``sligthly tall case'' is
            \begin{equation}
                \complexity{N_\mathrm{slice} \ell^3 + N_\mathrm{slice} N_m (k - \ell) \ell^2 +N_\mathrm{slice} N_m \ell^2}.
            \end{equation}
            The first term corresponds to the QZ reduction on the lower $\ell \times \ell$ subsection for each slice, the second term to the banded QR factorizations, whose bandwidth is $k - \ell$, and the third term to the inverse Lanczos iterations used to determine the minimum singular value.

        \subsubsection{Very tall case}

            When considering the ``very tall case'', $k \geq 2 \ell$, we again fix $m - 1$ parameters and reduce the problem to an one-parameter problem via slicing. 
            In this case, we first apply a QR factorization to the coefficient matrix corresponding to the unfixed parameter, thereby creating an upper triangular structure. 
            The QR factorization of the $k \times \ell$ coefficient matrix has a complexity of $\complexity{k \ell^2}$ and is performed only once, since for every slice $(\comp{\eigcomp[1]}, \comp{\eigcomp[2]}, \cdots, \comp{\eigcomp[m - 1]}) \in G_1 \times G_2 \times \cdots \times G_{m - 1}$, the one-parameter matrix pencil $(\comp{\eigcomp[1]}, \comp{\eigcomp[2]}, \ldots, \comp{\eigcomp[m - 1]}, \eigcomp[m])$ is multiplied by $\mt{Q}^\herm \in \Cset^{k \times k}$ to recover the upper triangular form. 
            Next, for every slice, $(\comp{\eigcomp[1]}, \comp{\eigcomp[2]}, \cdots, \comp{\eigcomp[m - 1]}) \in G_1 \times G_2 \times \cdots \times G_{m - 1}$, a QR factorization is applied to the lower rows of the remaining slice-dependent matrix in order to obtain the desired trapezoidal structure, resulting in $\prod_{i = 1}^{m - 1} \abs{G_i}$ QR factorizations of cost $\complexity{(k - \ell) \ell^2}$.
            Finally, for each grid point $\lambda_m \in G_m$, a banded QR factorization with $\complexity{\ell^3}$ floating-point operations is carried out before the inverse Lanczos iteration is used to determine the smallest singular value.
            The total computational complexity adds up to
            \begin{equation}
                \complexity{k \ell^2 + N_\mathrm{slice} (k - \ell) \ell^2 + N_\mathrm{slice} (k - \ell) \ell^2 + N_\mathrm{slice} N_m \ell^3 + N_\mathrm{slice} N_m \ell^2}.
            \end{equation}

\section{Perturbation theory in an application}
    \label{sec:application}

    Several system identification applications can be reformulated so that their globally optimal solution correspond to one eigenvalue of an rMEP.
    The least squares realization problem below is a prototypical example: 
    The first-order necessary conditions for optimality of the least squares minimization problem form a multivariate polynomial system, the common roots of which correspond to the eigenvalue of the associated rMEP~\cite{demoor2019least}, and the globally optimal solution is given by the eigenvalue that results into the minimal misfit.
    Here, we apply the concepts from the rectangular multispectral perturbation theory developed in the previous sections to such problems, and we report an empirical pattern via \cref{conj:optimality}, which links global optimality and eigenvalue conditioning.

    Given a sequence of output data $\vc{y} \in \Rset^N$, the least squares realization problem exists in finding (real) model parameters $\alpha_i \in \Rset$ and a slightly modified sequence of model-compliant output data $\vctilde{y} \in \Rset^N$ so that the model structure
    \begin{equation}
        \widehat{y}_k + \alpha_1 \widehat{y}_{k - 1} + \cdots + \alpha_m \widehat{y}_{k - m} = 0
    \end{equation}
    holds for all the modified data points $\widehat{y}_k$, thus for $k = m + 1, m + 2, \ldots, N$.
    The globally optimal model has parameters $\alpha_i$ such that the modification is minimal in the least squares sense, i.e., $\norm{\vchat{y} - \vc{y}}$ is minimal. 
    The corresponding optimization problem,
    \begin{equation}
        \begin{gathered}
            \min_{\vc{\alpha}, \vchat{y}} \frac{1}{2} \norm{\vchat{y} - \vc{y}} \\
            \sub \widehat{y}_k + \alpha_1 \widehat{y}_{k - 1} + \cdots + \alpha_m \widehat{y}_{k - m} = 0, \quad k = m + 1, m + 2, \ldots, N, 
        \end{gathered}
    \end{equation}
    can be solved by rephrasing it as an rMEP, the stationary points $(\vc{\alpha}, \vchat{y})$ constitute its eigenpairs $(\eig, \rvec)$, as explained in~\cite{demoor2019least}.

    \begin{example}
        \label{ex:application}

        We consider a model with two parameters $\alpha_1 = -0.5$ and $\alpha_2 = -0.5$. 
        Let the given sequence of output data be
        \begin{equation}
            \vc{y} = \begin{bmatrix}
                2.0000 \\
                4.0000 \\
                3.0000 \\
                3.5000 \\
                3.2500 
            \end{bmatrix}.
        \end{equation}
        Since we seek the model parameters, only real solutions are of interest.
        The computed real eigenvalues are listed in \cref{tab:application}, together with the norm-wise backward eigenvalue errors, norm-wise relative eigenvalue condition numbers, and cost values.
        The type of each real stationary point is determined from the cost function, visualized in \cref{fig:application:costfunction}.

        The condition numbers of the global optima (both the maxima and the minimum) appear to be lower than those of the other stationary points.
        This behavior is preferable, since we are primarily interested in the optimal models (here, the minimum misfit model).
        In particular, the global minimum is among the best-conditioned eigenvalues.
        Notice that the condition number of $\comp[(5)]{\eig}$, which is a saddle point far from the data-generating parameters, is several orders of magnitude larger than that of the other real stationary points.
        In the pseudospectrum (\cref{fig:application:pseudospectrum}), the pocket around this eigenvalue is more stretched, indicating a tangential intersection of the secular equations.
        As a side note, the condition numbers of the complex stationary points are also larger than those of the real optima.
    \end{example}

    \begin{table}[t]
        \centering
        \caption{Real stationary points of the least squares realization problem in \cref{ex:application}, expressed as eigenvalues $\comp{\eig}{(i)} = (\alpha_1, \alpha_2)$ of the associated rMEP. For each solution we report the eigenvalue condition number $\cond{\eig}$, the cost $\norm{\vctilde{y}} = \norm{\vchat{y} - \vc{y}}$, the backward error $\backward{\eig}$, and the type of stationary point. The global minimum at the data-generating parameters $(\alpha_1, \alpha_2) = (-0.5, -0.5)$ attains zero cost; the saddle point $\sol{\eig}{(5)}$, which lies far from those parameters, has a condition number several orders of magnitude larger than the remaining real stationary points, consistent with \cref{conj:optimality}.}
        \label{tab:application}
        \begin{tabular}{r|cccccc}
            \toprule
            & $\alpha_1$ & $\alpha_2$ & $\backward{\eig}$ & $\cond{\eig}$ & $\norm{\vctilde{y}}$ & type \\
            \midrule
            $\comp{\eig}_{(1)}$ & $\num{-0.5000}$ & $\num{-0.5000}$ & $\num{6.5e-14}$ & $\num{1.5e2}$ & $\phantom{0}\num{0.0000}$ & minimum \\
            $\comp{\eig}_{(2)}$ & $\num{-0.4506}$ & $\phantom{-}\num{0.9892}$ & $\num{2.5e-16}$ & $\num{1.0e2}$ & $\num{51.8125}$ & maximum \\
            $\comp{\eig}_{(3)}$ & $\phantom{-}\num{1.5275}$ & $\phantom{-}\num{0.6221}$ & $\num{6.5e-14}$ & $\num{1.7e2}$ & $\num{51.8125}$ & maximum \\
            $\comp{\eig}_{(4)}$ & $\phantom{-}\num{0.7851}$ & $\phantom{-}\num{0.7888}$ & $\num{5.0e-16}$ & $\num{2.6e2}$ & $\num{45.4285}$ & saddle point \\
            $\comp{\eig}_{(5)}$ & $\phantom{-}\num{5.7691}$ & $\num{-7.5333}$ & $\num{5.8e-16}$ & $\num{6.0e5}$ & $\phantom{0}\num{1.6429}$ & saddle point \\
            \bottomrule
        \end{tabular}
    \end{table}

    \begin{figure}
        \centering
        \begin{subfigure}{0.45\textwidth}
            \centering
            \includegraphics{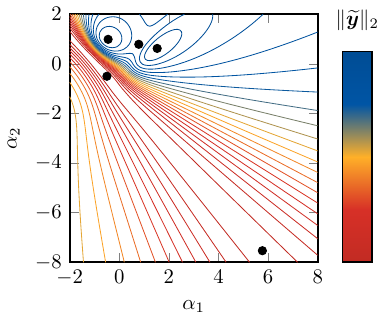}
            \caption{$\norm{\vctilde{y}} = \norm{\vchat{y} - \vc{y}}$}
            \label{fig:application:costfunction}
        \end{subfigure}
        \begin{subfigure}{0.45\textwidth}
            \centering
            \includegraphics{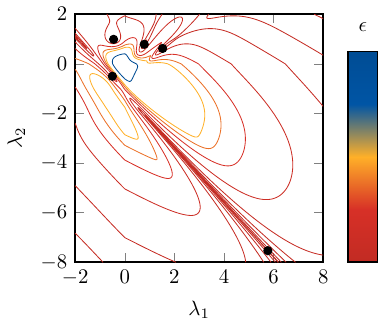}
            \caption{$\pseudo{\mep{\eig}}$}
            \label{fig:application:pseudospectrum}
        \end{subfigure}
        \caption{Contour lines of the cost function and pseudospectrum of the corresponding rMEP for the system identification application in \cref{ex:application}. The real stationary points are marked in both figures (~\ref{plot:solution}~). From the contour lines in \cref{fig:application:costfunction}, their type can be deducted, while the size of the pseudospectrum pockets in \cref{fig:application:pseudospectrum} gives an idea of the conditioning of every eigenvalue that corresponds to a stationary point.}
        \label{fig:application}
    \end{figure}

    Similar conclusions can be drawn when the given output data is corrupted with noise and is no longer model compliant.
    Motivated by the previous example and similar experiences for the globally optimal identification of autoregressive moving-average model parameters (a latency minimization problem)~\cite{vermeersch2019globally}, we give the following conjecture without proof and with some caution.

    \begin{conjecture}
        \label{conj:optimality}
        Consider the least squares optimization problem arising from linear time-invariant system identification, in which the stationary points correspond to the eigenvalues of an rMEP obtained via the first-order optimality conditions.
        The global optima of the underlying optimization problem are among the best-conditioned eigenvalues of the rMEP.
    \end{conjecture}

    When solving rMEPs in system identification, \cref{conj:optimality} suggests that the solutions of interest are well-conditioned, which is a favorable property.

\section{Conclusion and future work}
    \label{sec:conclusion}

    We provided a first systematic treatment of perturbation theory for the rectangular multiparameter eigenvalue problem (rMEP).
    By using perturbed coefficient matrices, we derived useful expressions to compute norm-wise backward errors and condition numbers for its eigenvalues and eigenvectors.
    Moreover, we analyzed the multiparameter pseudospectrum, which is intimately connected to eigenvalue sensitivity.
    Numerical examples illustrated and linked the different concepts.

    We showed that the pseudospectrum offers a global counterpart to the local sensitivity captured by the eigenvalue condition number.
    By \cref{cor:singularvalues}, the pseudospectrum directly visualizes how rapidly the minimal singular value grows away from each eigenvalue solution.
    Near a well-conditioned eigenvalue, the minimum singular value rises steeply and the pseudospectrum forms a tight pocket around the solution; near an ill-conditioned eigenvalue, the opposite occurs.
    The local rate is precisely what the condition number measures via \cref{prop:intersection}: when the secular equations intersect at a small angle, the auxiliary matrix from \cref{lemma:genericity} is nearly singular, the condition number is large, and the pocket of the pseudospectrum is correspondingly fat.
    The two descriptions---local by the condition number and global by the pseudospectrum---are thus two views of the same phenomenon, and the examples illustrated how all concepts in the paper are connected.

    Although we have attempted to give a thorough treatment of rectangular multiparameter perturbation theory, several research directions remain open for future work.
    Alternative norms, such as mixed subordinate $\alpha, \beta$ matrix norms~\cite{higham1998structured} and the chordal norm~\cite{stewart1990matrix}, may lead to a cleaner or more insightful analysis.
    Spherical projections of the pseudospectrum in the spirit of~\cite{higham2002more} offer another natural extension to our approach.
    Finally, certification techniques for multiparameter eigenvalues appear promising and deserve further investigation.

    \bibliographystyle{plain}
    \bibliography{bibliography}

    \appendix
\crefalias{section}{appendix}
\noindent{\Large\bfseries Appendices} \\

\noindent The paper comes with two appendices: an overview of the left null space and left eigenvector for rectangular eigenvalue problems (\cref{app:leftnullspace}) and an introduction to perturbation theory (\cref{app:perturbationtheory}).

\section{On the left eigenvectors of the matrix polynomial}
    \label{app:leftnullspace}

    Since the coefficient matrices are rectangular, there exist polynomial vectors $\lvec(\eig)$ such that $\lvec(\eig)^\herm \mep{\eig} \equiv \vc{0}$ for any $\eig \in \Cset^m$. 
    These vectors form the trivial left null space of the matrix polynomial $\mep{\eig}$, which has dimension $m - 1$ for the standard case where $k = l + m - 1$.
    As discussed in~\cite{khazanov1998spectral}, there exist a polynomial basis vectors $\lvec_i(\eig)$, for $i = 1, \ldots, m - 1$, for the trivial left null space.
  
    \begin{figure}
        \centering
        \tikzexternalenable
        \begin{tikzpicture}
    \begin{axis}[
        figurestyle,
        xmin = 0,
        xmax = 2,
        ymin = -1,
        ymax = 1,
        xlabel = {$t$},
        ylabel = {$\lvec(3)$}
    ]

        \addplot[blueline] table [row sep = crcr]{
            0	0.904534033733291\\
            0.02	0.905643538831005\\
            0.04	0.906779825994907\\
            0.06	0.90794379259295\\
            0.08	0.909136369952471\\
            0.1	0.910358524267297\\
            0.12	0.911611257421171\\
            0.14	0.912895607698934\\
            0.16	0.914212650350804\\
            0.18	0.915563497967638\\
            0.2	0.916949300616178\\
            0.22	0.918371245672467\\
            0.24	0.919830557278678\\
            0.26	0.921328495332824\\
            0.28	0.922866353901859\\
            0.3	0.924445458925614\\
            0.32	0.926067165051089\\
            0.34	0.927732851402656\\
            0.36	0.929443916052467\\
            0.38	0.931201768904959\\
            0.4	0.933007822647968\\
            0.42	0.934863481347774\\
            0.44	0.936770126173357\\
            0.46	0.938729097622154\\
            0.48	0.940741673480667\\
            0.5	0.942809041582063\\
            0.52	0.944932266211595\\
            0.54	0.947112246749232\\
            0.56	0.949349666814963\\
            0.58	0.951644931779876\\
            0.6	0.953998092005724\\
            0.62	0.956408748551914\\
            0.64	0.958875937310276\\
            0.66	0.961397986554589\\
            0.68	0.963972341673418\\
            0.7	0.966595349328283\\
            0.72	0.969261991365655\\
            0.74	0.971965556412608\\
            0.76	0.97469723408159\\
            0.78	0.9774456129529\\
            0.8	0.980196058819607\\
            0.82	0.982929943867198\\
            0.84	0.985623690297673\\
            0.86	0.988247583172436\\
            0.88	0.990764296755919\\
            0.9	0.993127066322842\\
            0.92	0.995277423429141\\
            0.94	0.997142397714068\\
            0.96	0.998631073964667\\
            0.98	0.99963038254143\\
            1.02	0.99956625458814\\
            1.04	0.998114984186316\\
            1.06	0.995383406994079\\
            1.08	0.991051274184318\\
            1.1	0.984731927834662\\
            1.12	0.975964444181277\\
            1.14	0.964208851210044\\
            1.16	0.948847472716111\\
            1.18	0.929196646622465\\
            1.2	0.904534033733291\\
            1.22	0.874146757476248\\
            1.24	0.83740361432108\\
            1.26	0.793849441986747\\
            1.28	0.743311116239434\\
            1.3	0.685994340570035\\
            1.32	0.622543017479467\\
            1.34	0.554034688851555\\
            1.36	0.48189987357506\\
            1.38	0.407776599618351\\
            1.4	0.333333333333334\\
            1.42	0.260102435504942\\
            1.44	0.189358320929341\\
            1.46	0.122055639539471\\
            1.48	0.0588235294117645\\
            1.5	-0\\
            1.52	-0.0543125446593571\\
            1.54	-0.104186452214125\\
            1.56	-0.149812850831677\\
            1.58	-0.19145550896473\\
            1.6	-0.229415733870562\\
            1.62	-0.264007392310479\\
            1.64	-0.295540231644524\\
            1.66	-0.324309518712743\\
            1.68	-0.350590224209229\\
            1.7	-0.374634324632678\\
            1.72	-0.39667014528604\\
            1.74	-0.416902970610721\\
            1.76	-0.435516386612325\\
            1.78	-0.452673997529931\\
            1.8	-0.468521285665818\\
            1.82	-0.483187471111309\\
            1.84	-0.496787287358783\\
            1.86	-0.509422627782115\\
            1.88	-0.521184042945983\\
            1.9	-0.532152084190191\\
            1.92	-0.542398498111707\\
            1.94	-0.551987281629653\\
            1.96	-0.560975609756098\\
            1.98	-0.569414649003076\\
            2	-0.577350269189626\\
            };
            \label{plot:symbolical}

        \addplot[redmark, only marks, mark = *, mark repeat = 3, mark phase = 3] table [row sep = crcr]{
            0	0.904534033733291\\
            0.02	0.905643538831005\\
            0.04	0.906779825994906\\
            0.06	0.90794379259295\\
            0.08	0.909136369952471\\
            0.1	0.910358524267297\\
            0.12	0.911611257421171\\
            0.14	0.912895607698934\\
            0.16	0.914212650350804\\
            0.18	0.915563497967639\\
            0.2	0.916949300616178\\
            0.22	0.918371245672467\\
            0.24	0.919830557278678\\
            0.26	0.921328495332824\\
            0.28	0.922866353901859\\
            0.3	0.924445458925614\\
            0.32	0.926067165051089\\
            0.34	0.927732851402656\\
            0.36	0.929443916052467\\
            0.38	0.931201768904959\\
            0.4	0.933007822647968\\
            0.42	0.934863481347774\\
            0.44	0.936770126173357\\
            0.46	0.938729097622154\\
            0.48	0.940741673480667\\
            0.5	0.942809041582063\\
            0.52	0.944932266211595\\
            0.54	0.947112246749232\\
            0.56	0.949349666814963\\
            0.58	0.951644931779875\\
            0.6	0.953998092005724\\
            0.62	0.956408748551914\\
            0.64	0.958875937310276\\
            0.66	0.961397986554589\\
            0.68	0.963972341673418\\
            0.7	0.966595349328283\\
            0.72	0.969261991365655\\
            0.74	0.971965556412608\\
            0.76	0.974697234081589\\
            0.78	0.9774456129529\\
            0.8	0.980196058819607\\
            0.82	0.982929943867198\\
            0.84	0.985623690297673\\
            0.86	0.988247583172436\\
            0.88	0.990764296755919\\
            0.9	0.993127066322841\\
            0.92	0.995277423429141\\
            0.94	0.997142397714068\\
            0.96	0.998631073964667\\
            0.98	0.99963038254143\\
            1	0\\
            1.02	0.99956625458814\\
            1.04	0.998114984186316\\
            1.06	0.995383406994079\\
            1.08	0.991051274184318\\
            1.1	0.984731927834662\\
            1.12	0.975964444181277\\
            1.14	0.964208851210044\\
            1.16	0.948847472716111\\
            1.18	0.929196646622466\\
            1.2	0.904534033733292\\
            1.22	0.874146757476247\\
            1.24	0.83740361432108\\
            1.26	0.793849441986747\\
            1.28	0.743311116239435\\
            1.3	0.685994340570034\\
            1.32	0.622543017479467\\
            1.34	0.554034688851554\\
            1.36	0.48189987357506\\
            1.38	0.407776599618352\\
            1.4	0.333333333333334\\
            1.42	0.260102435504942\\
            1.44	0.189358320929341\\
            1.46	0.122055639539471\\
            1.48	0.0588235294117648\\
            1.5	-0\\
            1.52	-0.0543125446593569\\
            1.54	-0.104186452214125\\
            1.56	-0.149812850831677\\
            1.58	-0.19145550896473\\
            1.6	-0.229415733870562\\
            1.62	-0.264007392310478\\
            1.64	-0.295540231644524\\
            1.66	-0.324309518712743\\
            1.68	-0.350590224209229\\
            1.7	-0.374634324632678\\
            1.72	-0.39667014528604\\
            1.74	-0.416902970610721\\
            1.76	-0.435516386612325\\
            1.78	-0.452673997529931\\
            1.8	-0.468521285665818\\
            1.82	-0.483187471111309\\
            1.84	-0.496787287358783\\
            1.86	-0.509422627782116\\
            1.88	-0.521184042945983\\
            1.9	-0.532152084190191\\
            1.92	-0.542398498111707\\
            1.94	-0.551987281629653\\
            1.96	-0.560975609756098\\
            1.98	-0.569414649003076\\
            2	-0.577350269189626\\
        };
        \label{plot:numerical}

        \addplot[solution] coordinates{
            (1, 0) 
            (1, 1)
        };
        \label{plot:solution}
    \end{axis}
\end{tikzpicture} 
        \tikzexternaldisable
        \caption{Visualization of the third component of the left null space $\mt{Y}$ for a parameter $t$, which parametrizes the eigenvalue as $\eig = \left(t, t\right)$. In $t = 1$, the parametrization corresponds to one of the eigenvalue solutions, i.e., $\sol{\eig} = (1, 1)$. Both the symbolically obtained polynomial vectors (\ref{plot:symbolical}) and the numerically obtained left null space vectors (~\ref{plot:numerical}~) are shown. The symbolic expression for $r(t)$ in~\eqref{eq:symbolicleftnullspace} is normalized so that the vector in~\eqref{eq:symbolicleftnullspace} has norm one. For an eigenvalue, which is here in $t = 1$, the rank of $\mep{\eig}$ drops and the dimension of the left null space equals two. The two components for the eigenvalue solution are highlighted in black (~\ref{plot:solution}~).}
        \label{fig:leftnullspace}
    \end{figure}
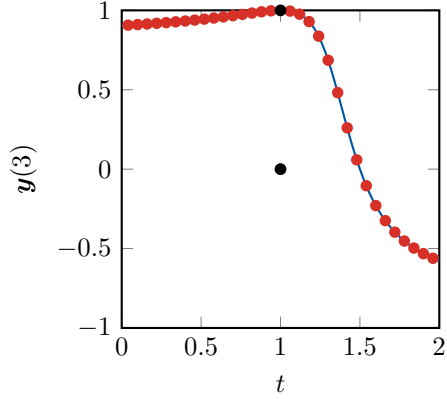

    For an eigenvalue $\sol{\eig}$, $\sol{\rho}$ additional vectors in the left null space appear, which we call the left eigenvectors $\sol{\lvec} \in \Cset^k$.
    When the eigenvalues are simple (as is assumed in the paper), there is exactly one such left eigenvector.

    \begin{example}
        The problem in \cref{ex:runningexample} also has a trivial vector in the left null space:
        \begin{equation}
            \lvec^\herm = 
            \begin{bmatrix}
                p(\eigcomp[1], \eigcomp[2]) & q(\eigcomp[1], \eigcomp[2]) & r(\eigcomp[1], \eigcomp[2])
            \end{bmatrix},
        \end{equation}
        where it is possible to determine that these polynomial vectors are
        \begin{equation}
            \label{eq:symbolicleftnullspace}
            \begin{aligned}
                p(\eigcomp[1], \eigcomp[2]) &= -2 (\eigcomp[1] - 1) (\eigcomp[2] - 1), \\
                q(\eigcomp[1], \eigcomp[2]) &= -3 \eigcomp[1] + 7 \eigcomp[2] - \eigcomp[1] \eigcomp[2] + \eigcomp[1]^2 - 2 \eigcomp[2]^2 - 2, \\
                r(\eigcomp[1], \eigcomp[2]) &= 2 (\eigcomp[2] - 1) (\eigcomp[1] + \eigcomp[2] - 3) 
            \end{aligned}
        \end{equation}
        In the eigenvalues, the dimension of the left null space equals two and there exists a true left eigenvector. 
        The different components of the left null space are visualized in \cref{fig:leftnullspace} for the eigenvalue varying from one solution to another.
    \end{example}

    While the polynomial vectors in the above-mentioned example have been obtained by performing symbolic computations by hand, there exist more systematic approaches, for example, the ones described in~\cite{khazanov2000generating}.

\section{On the backward error and condition number}
    \label{app:perturbationtheory}

    We collect the standard definitions of perturbation theory here, so that it is clear from which definitions the analysis in the paper starts.
    More information can be found in detailed textbooks, such as~\cite{trefethen1997numerical}.

    Let $f: X \to Y$ be a map between normed vector spaces, modeling the problem of mapping input data $x \in X$ to a solution $y = f(x) \in Y$.
    Similarly, an algorithm can be viewed as another map $\widetilde{f}: X \to Y$ between the same two spaces.
    We say that an algorithm $\widetilde{f}$ for a problem $f$ is backward stable if for each $x \in X$ the backward error is of the order of machine precision.
    \begin{definition}
        Let $\widetilde{f}$ be an algorithm for the problem $f$.
        The \emph{norm-wise relative backward error} of $\widetilde{f}$ at $x \in X$ is
        \begin{equation}
            \label{eq:backwarderror}
            \eta \mleft( x \mright) = \inf\left\lbrace \frac{\mleft\Vert \Delta x \mright\Vert}{\mleft\Vert x \mright\Vert} : \widetilde{f}(x) = f(x + \err x)\right\rbrace,
        \end{equation}
        with the convention that $\eta\mleft(x\mright) = \infty$ when no such $\err x$ exists.
        By removing $\left\Vert x \right\Vert$ from the denominator, $\eta\mleft(x\mright)$ is the \emph{norm-wise absolute backward error}.
    \end{definition}

    The condition number of a function measures how much the output value of the function can change for a small change in the input argument. 
    \begin{definition}
        Let $f$ be a problem and $x \in X$ a point at which $f$ is defined.
        The \emph{norm-wise absolute and relative condition numbers} of $f$ at $x$ are
        \begin{gather}
            \widehat{\kappa} \mleft( x \mright) = \lim_{\delta \to 0} \sup_{\mleft\Vert \delta x \mright\Vert \leq \delta} \frac{\mleft\Vert \delta f \mright\Vert}{\mleft\Vert \delta x \mright\Vert}, \\
            \kappa \mleft( x \mright) = \lim_{\delta \to 0} \sup_{\mleft\Vert \delta x \mright\Vert \leq \delta} \mleft( \frac{\mleft\Vert \delta f \mright\Vert}{\mleft\Vert f (x) \mright\Vert} / \frac{\mleft\Vert \delta x \mright\Vert}{\mleft\Vert x \mright\Vert} \mright).
        \end{gather}
        When $f$ is multi-valued, the condition number is defined relative to a fixed smooth branch.
    \end{definition}
    
    The condition number may depend strongly on $x$, meaning that some instances of the problem are more sensitive to perturbations than others.
    Both absolute and relative condition numbers have their uses, but the latter are more important in numerical analysis because the floating-point arithmetic used by computers introduces relative errors rather than absolute ones.

\end{document}